\documentclass[a4paper,fleqn]{cas-sc}

\usepackage[numbers]{natbib}

\usepackage[utf8]{inputenc} 
\usepackage{accents}				
\usepackage{todonotes}
\usepackage{subcaption}
\usepackage{stmaryrd} 
\usepackage{amsthm} 
\usepackage[makeroom]{cancel} 
\usepackage{bookmark} 
\usepackage{empheq}
\usepackage{ulem}

\newcommand{\mat}[1]{\underline{\mathbf{#1}}}
\newcommand\acclrvec[1]{\accentset{\,\leftrightarrow}{#1}}	
\newcommand{\blocktensor}[1]{\acclrvec{{\mathbf #1}}}	
\newcommand\blocktensorG[1]{\acclrvec{\boldsymbol #1}}	
\newcommand{\Nabla} {\vec{\nabla}}
\newcommand{\numfluxb}[1]{\hat{\mathbf{#1}} }

\newcommand\threeMatrix[1]{\underline{ #1}}				
			
\def\d{\mathrm{d}}

\newcommand{\bigpartialderiv}[2]{ \frac{\partial {#1}}{\partial {#2} } }

\newcommand\stateG[1]{\boldsymbol #1}			

\newcommand{\noncon}{\stateG{\Upsilon}}	

\newcommand\state[1]{\mathbf{#1}}

\newcommand{\supEuler}{{\mathrm{Euler}}}
\newcommand{\supMHD}{{\mathrm{MHD}}}
\newcommand{\supGLM}{{\mathrm{GLM}}}
\newcommand{\SBP}{{\mathrm{SBP}}}
\newcommand{\FV}{{\mathrm{FV}}}

\newcommand{\Jan}{\stateG{\Phi}}

\newcommand{\phiMHD}{\stateG{\phi}^\supMHD}

\newcommand{\phiGLM}{\blocktensorG{\phi}^\supGLM}
\newcommand{\phiGLMs}{\stateG{\phi}^\supGLM}

\newcommand{\numnonconsD}[1]{ #1^{\Diamond} }
\newcommand{\numnonconsS}[1]{ #1^{\star} }

\newcommand{\avg}[1]{\left\{\hspace*{-3pt}\left\{#1\right\}\hspace*{-3pt}\right\}}

\newcommand{\numnonconsSxi}[1]{ #1^{1\star} }
\newcommand{\numnonconsSeta}[1]{#1^{2\star} }
\newcommand{\numnonconsSzeta}[1]{ #1^{3\star} }

\def\NN{\mathcal{N}} 

\newtheorem{proposition}{Proposition}

\newtheorem{remark}{Remark}

\begin{document}

\let\WriteBookmarks\relax
\def\floatpagepagefraction{1}
\def\textpagefraction{.001}

\shorttitle{A Flux-Differencing SBP Formula for Non-Conservative Systems: Subcell Limiting}
\shortauthors{Rueda-Ramírez \& Gassner}

\title [mode = title]{A Flux-Differencing Formula for Split-Form Summation By Parts Discretizations of Non-Conservative Systems\\
{\normalsize Applications to Subcell Limiting for magneto-hydrodynamics}
}

\author[1]{Andrés M. Rueda-Ramírez}[
                        orcid=0000-0001-6557-9162]
\cormark[1]
\ead{aruedara@uni-koeln.de}

\credit{Conceptualization, Methodology, Software, Validation, Formal analysis, Data Curation, Writing - Original Draft, Visualization}

\author[1,4]{Gregor J. Gassner}
[orcid=0000-0002-1752-1158]
\credit{Conceptualization, Methodology, Validation, Formal analysis, Data Curation, Writing - Original Draft}

\address[1]{Department of Mathematics and Computer Science, University of Cologne, Weyertal 86-90, 50931 Cologne, Germany}

\address[4]{Center for Data and Simulation Science, University of Cologne, 50931 Cologne, Germany}

\cortext[cor1]{Corresponding author}

\begin{abstract}
In this paper, we show that diagonal-norm summation by parts (SBP) discretizations of general non-conservative systems of hyperbolic balance laws can be rewritten as a finite-volume-type formula, also known as flux-differencing formula, if the non-conservative terms can be written as the product of a local and a symmetric contribution.
Furthermore, we show that the existence of a flux-differencing formula enables the use of recent subcell limiting strategies to improve the robustness of the high-order discretizations.

To demonstrate the utility of the novel flux-differencing formula, we construct hybrid schemes that combine high-order SBP methods (the discontinuous Galerkin spectral element method and a high-order SBP finite difference method) with a compatible low-order finite volume (FV) scheme at the subcell level.
We apply the hybrid schemes to solve challenging magnetohydrodynamics (MHD) problems featuring strong shocks.
\end{abstract}




\maketitle

\section{Introduction}

Given a set of ordered nodes $\state{X} = \{x_0, x_1, \ldots, x_N\}$, $x_i \in \mathbb{R}$, two vectors of solution values at the nodes $\state{U}, \state{V}
\in \mathbb{R}^{N+1}$, a one-dimensional differentiation operator $\mat{D} \in \mathbb{R}^{N+1,N+1}$ and a positive definite matrix $\mat{M} \in \mathbb{R}^{N+1,N+1}$ (called norm or mass matrix), which defines an $L_2$ inner product in the interval $[x_0,x_N]$,
\begin{equation} \label{eq:SBP_diff_int}
\frac{\partial \state{U}}{\partial x} \approx \mat{D} \state{U},
\qquad
\int_{x_0}^{x_N} u \, v \, \d x \approx  \state{U}^T \mat{M} \state{V},
\end{equation}
the combination of $\mat{M}$ and $\mat{D}$ is defined as a first-derivative summation by parts (SBP) operator if it fulfills the properties
\begin{equation} \label{eq:SBP_property}
    \mat{Q} \coloneqq \mat{M}\,\mat{D},
    \qquad
    \mat{Q} + \mat{Q}^T = \mat{B},
\end{equation}
where $\mat{B} \coloneqq \mathrm{diag} ([-1, 0, \ldots, 0, 1]) \in \mathbb{R}^{N+1,N+1}$ is the so-called boundary evaluation matrix.


The basic theory of SBP operators was introduced by \citet{kreiss1974finite} in their search for provably stable high-order accurate finite difference (FD) methods for linear advection problems.
They introduced the SBP property \eqref{eq:SBP_property} as a discrete equivalent of integration by parts, which allowed them to mimic the continuous analysis to bound the energy norm of the solution.
After the work of \citet{kreiss1974finite}, SBP operators have been further developed and used for a broad range of applications \cite{gustafsson1975convergence,olsson1994high,strand1994summation,nordstrom2001high,carpenter1993stability,Fisher2013a,Carpenter2014,Gassner2013}.

1D SBP operators can be extended to multiple space dimensions using tensor product expansions (see, e.g., \cite{nordstrom2001high}).
For periodic problems in Cartesian meshes, the SBP property \eqref{eq:SBP_property} is sufficient to prove stability.
For non-periodic problems or problems where multi-element meshes are needed, \citet{funaro1987domain} proposed the use of penalty terms for the inter-element coupling and the weak imposition of boundary conditions.
\citet{carpenter1993stability} elaborated on these methods and introduced the simultaneous approximation term (SAT), which became the standard way to treat boundary terms and interfaces for FD methods based on SBP operators, e.g.,  \cite{fernandez2014review,del2019extension}.

\citet{Gassner2013} showed that the discontinuous Galerkin spectral element method (DGSEM)\footnote{The DGSEM is a nodal collocated variant of the discontinuous Galerkin method, which uses Legendre-Gauss or Legendre-Gauss-Lobatto nodes to represent the solution (via Lagrange interpolating polynomials) and to perform numerical integration.} \cite{black1999conservative,kopriva2009implementing} that uses Legendre-Gauss-Lobatto (LGL) points is an SBP operator with a diagonal norm and that the standard DG surface integral with numerical flux functions is a special case of SAT boundary treatments. 
This equivalence allowed to apply the stability theory of SBP operators to the DGSEM framework.
Moreover, the similarities between the LGL-DGSEM method and high-order FD methods based on SBP operators enable the use of Riemann solvers to do the inter-element coupling and the imposition of boundary conditions also for Finite Difference SBP operators, a usual practice in the DG community.
In fact, the different variants of the SAT technique can be interpreted as choices of numerical fluxes to approximate the solution to the Riemann problem at the element interfaces.

For nonlinear systems of conservation laws, the ability of SBP operators to mimic integration by parts at the discrete level has been exploited to construct entropy-stable discretizations \cite{Fisher2013,Fisher2013a,Carpenter2014} and other nonlinearly stable discretizations based on split forms of the governing equations \cite{Gassner2013,Gassner2016} that ensure kinetic energy preservation or pressure equilibrium.

\citet{Fisher2013a} showed that diagonal-norm SBP discretizations of systems of conservation laws can be written as a finite-volume-type formula.
The existence of a flux-differencing formula guarantees local (node-wise) conservation in the sense of Lax-Wendroff.
Moreover, they showed \cite{Fisher2013a,Carpenter2014} that local entropy conservation/dissipation can also be guaranteed with the right choice of numerical volume fluxes.\\

High-order schemes can produce nonphysical oscillations in the presence of discontinuities of the solution, e.g., shocks.
Moreover, shocks and under-resolved turbulence can lead to nonphysical negative values of some solution quantities \cite{Rueda-Ramirez2020}, such as density or pressure.
To deal with this issues, several stabilization techniques based on artificial viscosity \cite{hiltebrand2014entropy,Persson2006,Klockner2011} or the combination of the high-order scheme with a more robust low-order scheme, see e.g.,  \cite{Vilar2019,Sonntag2017,Hennemann2020,Rueda-Ramirez2020,Rueda-Ramirez2021,kuzmin2020monolithic,kuzmin2010failsafe,hajduk2021monolithic,guermond2019invariant}.
Among the latter class of techniques, the hybrid DGSEM/FV scheme of \citet{Hennemann2020},which computes a convex combination of the high-order and the low-order operators in an element-wise manner, is an interesting choice because it can be used with any SBP operator, has been used to discretize general non-conservative systems of balance laws \cite{Rueda-Ramirez2020}, is able to handle strong shocks efficiently \cite{RUEDARAMIREZ2022}, and can be used to impose bounds on physical quantities \cite{Rueda-Ramirez2021}.

In a recent work \cite{RUEDARAMIREZ2022}, we showed that the flux-differencing formula of \citet{Fisher2013a} is not only useful to show local conservation, but it can also be used to construct bounds-preserving methods using subcell-wise limiting (as in, e.g. \cite{Pazner2020}).
As opposed to the element-wise limiting technique \cite{Hennemann2020,Rueda-Ramirez2021,Rueda-Ramirez2020}, the subcell-wise limiting technique computes node-local convex combinations of the high- and low-order operators, leading to a very local addition of dissipation.
However, our subcell-wise blending procedure \cite{RUEDARAMIREZ2022} is restricted to systems of conservation laws since the flux-differencing formula of \citet{Fisher2013a} holds only for this kind of systems.

In this work, we show that discretizations of non-conservative systems of balance laws that use diagonal-norm SBP operators can also be written as a flux-differencing formula.
Further, we use the novel flux-differencing form of the equation to formulate hybrid FV/SBP discretizations which are bounds-preserving.
We illustrate the advantages of the new flux-differencing formula to formulate bounds-preserving high-order discontinuous Galerkin and finite difference discretizations of the non-conservative magnetohydrodynamics equations augmented with a generalized Lagrange multiplier hyperbolic divergence cleaning technique (the GLM-MHD equations).
These high-order discretizations are stabilized by combining them locally (at the node level) with a robust FV scheme, which allows to do \textit{a-posteriori} limiting, e.g., imposing maximum principle and entropy constraints, and \textit{a-priori} limiting using feature-based indicators.

\section{Numerical Methods}

In this work, we consider general nonlinear systems of balance laws of the form
\begin{equation} \label{eq:noncons-system}
\bigpartialderiv{}{t} \state{u} 
+ \bigpartialderiv{}{x} \blocktensor{f} (\state{u}) 
+ \underbrace{
    \stateG{\phi} (\state{u}) \circ \bigpartialderiv{\state{b}}{x}
}_{\noncon}
= \state{0},
\end{equation}
where 
$\state{u}$ is a state vector, $\blocktensor{f}^a$ is a possibly nonlinear advective flux function, and $\noncon$ is a non-conservative term, which can be written as the Hadamard (element-wise) product of the state vector $\stateG{\phi}(\state{u})$ and the gradient of $\state{b}$, a spatially varying quantity that may depend on $\state{u}$.
Many equations can be written in this form, including the shallow water equations with variable bottom topography, the entropy consistent magnetohydrodynamics equations \cite{Powell2001,Dedner2002,Derigs2017}, the Baer-Nunziato equations for multiphase flow \cite{baer1986two}, among others.

For compactness and readability, in this section, we present numerical discretization schemes for the one-dimensional non-conservative system of balance laws \eqref{eq:noncons-system}. 
The schemes for three-dimensional non-conservative systems on general curvilinear meshes can be obtained using tensor-product expansions and are summarized for completeness in Appendix \ref{app:3D}.

\subsection{High-Order SBP Discretization} \label{sec:GLM-MHD}

To approximate the solution of \eqref{eq:noncons-system}, we tessellate the domain into non-overlapping elements and allow the solution to be discontinuous across element interfaces.
In each element, we use a generalized version of a diagonal-norm SBP operator with two-point fluxes
\cite{Fisher2013a,Gassner2013,Carpenter2014,Gassner2016}, where we include a non-symmetric two-point term to discretize the non-conservative terms of the equation.
The diagonal-norm SBP discretization of \eqref{eq:noncons-system} for the degree of freedom $j \in [0,N]$ reads \cite{rueda2022entropy}
\begin{empheq}{align} \label{eq:DGSEM} 
m_j \dot{\state{u}}^{\SBP}_j 
+ &
\underbrace{
\sum_{k=0}^N S_{jk} \left( \state{f}^{*}_{(j,k)} + \numnonconsS{\Jan}_{(j,k)} \right)
}_{\mathrm{Volume \, term}}
- 
\underbrace{
\delta_{j0} \left( \numfluxb{f}_{(0,L)} + \numnonconsD{\Jan}_{(0,L)} \right)
+ \delta_{jN} \left( \numfluxb{f}_{(N,R)} + \numnonconsD{\Jan}_{(N,R)} \right)
}_{\mathrm{Surface \,  term}}
= \state{0},
\end{empheq}
where $m_j$ is the entry of the diagonal mass matrix that corresponds to node $j$, and $\delta_{ij}$ denotes Kronecker's delta function with node indexes $i$ and $j$. For simplicity, we absorb the element geometry mapping Jacobian $J$ in the definition of the mass matrix.
As we have shown in \cite{rueda2022entropy}, \eqref{eq:DGSEM} is algebraically equivalent to the high-order LGL-DGSEM schemes proposed by different authors (e.g.,  \cite{Bohm2018,Renac2019,coquel2021entropy,rueda2022entropy}) to discretize non-conservative systems.
In this work, we choose the current notation, as it will be useful to derive the flux-differencing formula.


The discretization is separated in volume and surface terms terms. 
In the surface term, $\numfluxb{f}$ denotes the surface numerical flux function, or SAT term, and $\numnonconsD{\Jan}_{(j,k)}$ denotes the surface numerical non-conservative term.
These terms are typically based on approximate Riemann solvers and thus depend on the local values and the values from the neighbor elements on the left ($L$) and right ($R$).

The volume term is computed by applying the skew-symmetric derivative matrix $\mat{S}\coloneqq 2\mat{Q}-\mat{B} = \mat{Q} - \mat{Q}^T$ to the so-called volume numerical flux ${\state{f}}^{*}_{(j,k)}$, a two-point flux function evaluated between nodes $j$ and $k$ that needs to be consistent with the continuous flux and symmetric in its two arguments, and the so-called volume numerical non-conservative term $\numnonconsS{\Jan}_{(j,k)}$, a non-symmetric two-point term evaluated between nodes $j$ and $k$ that needs to be consistent with $\stateG{\Phi} := \stateG{\phi}(\state{u}) \circ \state{b}$. The particular choice of numerical volume fluxes allow to generate flexible split-formulations of the non-linear PDE terms that help for de-aliasing and may even enable provable entropy stability, e.g., \cite{Fisher2013,Fisher2013a,Carpenter2014,Gassner2013,Gassner2016,Renac2019}.

\subsection{Flux-Differencing Formula} \label{sec:fluxdiff}

\begin{proposition}\label{prop:fluxdiff}
It is possible to rewrite \eqref{eq:DGSEM} as a flux-differencing formula,
\begin{equation} \label{eq:fluxdiff}
    m_j \dot{\state{u}}^{\SBP}_j = 
    \stateG{\Gamma}^{\SBP}_{(j,j-1)}
    -\stateG{\Gamma}^{\SBP}_{(j,j+1)}
    ,
    \qquad
    j=0, \ldots, N,
\end{equation}
where the indexes $j=-1$ and $j=N+1$ refer to the outer states (across the left and right boundaries, respectively) and $\stateG{\Gamma}^{\SBP}_{(j,k)}$ is the so-called staggered (or telescoping) ``flux'' between node $j$ and the \textbf{adjacent} node $k$,
if it is possible to write the volume numerical non-conservative term as a product of a local and a symmetric contribution,
\begin{equation} \label{eq:condition}
    \numnonconsS{\Jan}_{(j,k)} := \Jan^{\mathrm{loc}}_j \circ \, \Jan^{\mathrm{sym}}_{(j,k)},
\end{equation}
where $\Jan^{\mathrm{loc}}_j := \Jan^{\mathrm{loc}} (\state{u}_j,\state{b}_j)$ only depends on local quantities, and  $\Jan^{\mathrm{sym}}_{(j,k)} := \Jan^{\mathrm{sym}}_{(k,j)}$ is a symmetric two-point flux that depends on values at nodes $j$ and $k$.

The staggered fluxes are then defined as
\begin{align}
\stateG{\Gamma}^{\SBP}_{(0,-1)}  &= \numfluxb{f}_{(0,L)} + \numnonconsD{\Jan}_{(0,L)}
\label{eq:leftFlux}
\\
\stateG{\Gamma}^{\SBP}_{(j,k)} &= \sum_{l=0}^{\min(j,k)} \sum_{m=0}^N S_{lm} \state{f}^{*}_{(l,m)} 
+ \Jan^{\mathrm{loc}}_j \circ \sum_{l=0}^{\min(j,k)} \sum_{m=0}^N S_{lm} \Jan^{\mathrm{sym}}_{(l,m)}, & j=0, \ldots, N-1, k \in \{j-1,j+1\},
\label{eq:internFlux}\\
\stateG{\Gamma}^{\SBP}_{(N,N+1)} &= \numfluxb{f}_{(N,R)} +  \numnonconsD{\Jan}_{(N,R)}.
\label{eq:rightFlux}
\end{align}
\end{proposition}

\begin{proof}
It suffices to evaluate \eqref{eq:fluxdiff} for the boundary nodes, $j=0$ and $j=N$, and for an internal node $j \notin \{0,N\}$.

For the left boundary, $j=0$, we obtain
\begin{align*}
    m_0 \dot{\state{u}}^{\SBP}_0 &= 
    \stateG{\Gamma}^{\SBP}_{(0,-1)}
    -\stateG{\Gamma}^{\SBP}_{(0,1)} \\
    &=
    \numfluxb{f}_{(0,L)} + \numnonconsD{\Jan}_{(0,L)}
    - 
    \sum_{m=0}^N S_{0m} \state{f}^{*}_{(0,m)} 
    - \Jan^{\mathrm{loc}}_j \circ     \sum_{m=0}^N S_{0m}     \Jan^{\mathrm{sym}}_{(0,m)}
    \\
\mathrm{(using~eq.~\eqref{eq:condition})} \qquad
    &=
    \numfluxb{f}_{(0,L)} + \numnonconsD{\Jan}_{(0,L)}
    - 
    \sum_{m=0}^N S_{0m} \left( \state{f}^{*}_{(0,m)} + \numnonconsS{\Jan}_{(0,m)} \right),
\end{align*}
which is exactly equivalent to \eqref{eq:DGSEM} for $j=0$.

For the right boundary, $j=N$, we obtain
\begin{align*}
    m_N \dot{\state{u}}^{\SBP}_N &= 
    \stateG{\Gamma}^{\SBP}_{(N,N-1)}
    -\stateG{\Gamma}^{\SBP}_{(N,N+1)} \\
    =&
    \sum_{l=0}^{N-1} \sum_{m=0}^N S_{lm} \state{f}^{*}_{(l,m)} 
+ \Jan^{\mathrm{loc}}_j \circ \sum_{l=0}^{N-1} \sum_{m=0}^N S_{lm} \Jan^{\mathrm{sym}}_{(l,m)}
    - \numfluxb{f}_{(N,R)}  
    - \numnonconsD{\Jan}_{(N,R)}
    \\
    =&
    \sum_{l=0}^{N} \sum_{m=0}^N S_{lm} \state{f}^{*}_{(l,m)} 
    + \Jan^{\mathrm{loc}}_j \circ \sum_{l=0}^{N} \sum_{m=0}^N S_{lm} \Jan^{\mathrm{sym}}_{(l,m)}
    \\
    &-\sum_{m=0}^N S_{Nm} \state{f}^{*}_{(N,m)} 
    - \Jan^{\mathrm{loc}}_j \circ \sum_{m=0}^N S_{Nm} \Jan^{\mathrm{sym}}_{(N,m)}
    - \numfluxb{f}_{(N,R)}  
    - \numnonconsD{\Jan}_{(N,R)}
    \\
\mathrm{(Skew-symmetry~of~\mat{S}~\&~eq.~\eqref{eq:condition})} \qquad
    =&
    \frac{1}{2} \sum_{l=0}^{N} \sum_{m=0}^N (S_{lm} - S_{ml}) \state{f}^{*}_{(l,m)} 
    + \frac{1}{2} \Jan^{\mathrm{loc}}_j \circ \sum_{l=0}^{N} \sum_{m=0}^N (S_{lm} - S_{ml}) \Jan^{\mathrm{sym}}_{(l,m)}
    \\
    &-\sum_{m=0}^N S_{Nm} \left( \state{f}^{*}_{(N,m)} + \Jan^{\star}_{(N,m)} \right)
    - \numfluxb{f}_{(N,R)}  
    - \numnonconsD{\Jan}_{(N,R)}
    \\
\mathrm{(re-index~\&~symmetry~of~\state{f}^*,\Jan^{\mathrm{sym}})} \qquad
    =&-\sum_{m=0}^N S_{Nm} \left( \state{f}^{*}_{(N,m)} + \Jan^{\star}_{(N,m)} \right)
    - \numfluxb{f}_{(N,R)}  
    - \numnonconsD{\Jan}_{(N,R)},
\end{align*}
which is again exactly equivalent to \eqref{eq:DGSEM} for $j=N$.

Finally, for an arbitrary internal degree of freedom, $j \notin \{0,N\}$, we obtain
\begin{align*}
    m_j \dot{\state{u}}^{\SBP}_j =&
    \stateG{\Gamma}^{\SBP}_{(j,j-1)}
    -\stateG{\Gamma}^{\SBP}_{(j,j+1)}
    \\
    =&
    \sum_{l=0}^{j-1} \sum_{m=0}^N S_{lm} \state{f}^{*}_{(l,m)} 
    + \Jan^{\mathrm{loc}}_j \circ \sum_{l=0}^{j-1} \sum_{m=0}^N S_{lm} \Jan^{\mathrm{sym}}_{(l,m)}
    - \sum_{l=0}^{j} \sum_{m=0}^N S_{lm} \state{f}^{*}_{(l,m)} 
    - \Jan^{\mathrm{loc}}_j \circ \sum_{l=0}^{j} \sum_{m=0}^N S_{lm} \Jan^{\mathrm{sym}}_{(l,m)}
    \\
    =&
    - \sum_{m=0}^N S_{jm} \state{f}^{*}_{(j,m)} 
    - \Jan^{\mathrm{loc}}_j \circ \sum_{m=0}^N S_{jm} \Jan^{\mathrm{sym}}_{(j,m)}
    \\
\mathrm{(using~eq.~\eqref{eq:condition})} \qquad
    =&
    -
    \sum_{m=0}^N S_{jm} \left( \state{f}^{*}_{(j,m)} + \numnonconsS{\Jan}_{(j,m)} \right),
\end{align*}
which is the desired result.
\end{proof}

\begin{remark}
In the absence of non-conservative terms, $\noncon = \stateG{\Phi} = 0$, equations \eqref{eq:leftFlux}-\eqref{eq:rightFlux} reduce to the telescoping fluxes proposed by \citet{Fisher2013a}, which are symmetric, $\stateG{\Gamma}^{\SBP}_{(j,k)} = \stateG{\Gamma}^{\SBP}_{(k,j)} $, i.e., locally conservative at the node level.
\end{remark}

\begin{remark}
The terms $\stateG{\Gamma}^{\SBP}_{(j,k)}$ and $\stateG{\Gamma}^{\SBP}_{(k,j)} $ are in general not symmetric and, therefore, not unique for each interface between nodes. 
They account for the non-conservative nature of the system at each subcell interface.
\end{remark}

\subsection{Subcell Limiting}

Now that the high-order SBP scheme has been written as a flux-differencing formula, we follow the strategies presented in \cite{Hennemann2020,Pazner2020,Rueda-Ramirez2021,Rueda-Ramirez2020,RUEDARAMIREZ2022} and write a convex combination of the high-order SBP scheme with a robust low-order FV method.

The convex combination can be done in an element-wise manner,
\begin{equation} \label{eq:Elem_blend}
\dot{\state{u}}_{j} = (1-\alpha) \dot{\state{u}}_{j}^{\SBP} + \alpha \dot{\state{u}}_{j}^{\FV} \quad \forall j,
\end{equation}
where $\alpha$ is a piece-wise constant blending coefficient that is selected independently for each element.
However, the new flux-differencing formula enables the use of subcell-wise limiting for high-order SBP discretizations of non-conservative systems,
\begin{equation} \label{eq:Subcell_blend}
m_j \dot{\state{u}}_{j} =
\left(
  \stateG{\Gamma}_{(j-1,j)}
- \stateG{\Gamma}_{(j,j+1)}
\right),
\end{equation}
where each interface flux $\stateG{\Gamma}_{(\cdot,\cdot)}$ is computed as a convex combination of DG and FV fluxes,
\begin{equation}\label{eq:Subcell_blend2}
    \stateG{\Gamma}_{(a,b)} = (1-\alpha_{(a,b)}) \stateG{\Gamma}^{\SBP}_{(a,b)} + \alpha_{(a,b)} \stateG{\Gamma}^{\FV}_{(a,b)}.
\end{equation}
An individual blending coefficient $\alpha_{(a,b)}$ can be selected for each interface between any two adjacent nodes, $a$ and $b$, to impose positivity, entropic or non-oscillatory constraints to the numerical scheme.
With this subcell limiting strategy, every element has a collection of local blending factors $\alpha_{(a,b)}$ that determine locally the amount of FV mixed to the high-order SBP scheme. We refer to \cite{RUEDARAMIREZ2022} for more details on the options available for the blending.

\section{Application}

In this section, we illustrate the utility of the
novel flux-differencing formula to formulate robust high-order SBP discretizations of the entropy-consistent magneto-hydrodynamics equations augmented with a generalized Lagrange multipliers hyperbolic divergence cleaning technique (the GLM-MHD equations).
We formulate element- and subcell-wise combinations of the high-order SBP methods with a first-order FV scheme and test their robustness and dissipation properties with two popular benchmark examples for the MHD equations, the Orszag-Tang vortex and the MHD rotor test.

\subsection{The GLM-MHD System}

In this work, we use the variant of the ideal GLM-MHD equations that is consistent with the continuous entropy analysis of \citet{Derigs2018}.
The system of equations reads 
\begin{equation} \label{eq:GLM-MHD}
\partial_t \mathbf{u} 
+ \Nabla \cdot \blocktensor{f}^a (\mathbf{u}) 
+ \noncon(\mathbf{u}, \Nabla \mathbf{u})
= \state{0},
\end{equation}
with the state vector $\state{u} = (\rho, \rho \vec{v}, \rho E, \vec{B}, \psi)^T$, the advective flux $\blocktensor{f}^a$, 
and the non-conservative term $\noncon$.
Here, $\rho$ is the density, $\vec{v} = (v_1, v_2, v_3)^T$ is the velocity, $E$ is the specific total energy, $\vec{B} = (B_1, B_2, B_3)^T$ is the magnetic field, and $\psi$ is the so-called \textit{divergence-correcting field}, a generalized Lagrange multiplier (GLM) that is added to the original MHD system to minimize the magnetic field divergence. 
While these equations do not enforce the divergence-free condition exactly, $\Nabla \cdot \vec{B} = 0$, they evolve towards a divergence-free state \cite{Munz2000,Dedner2002,Derigs2018}.

The advective flux contains Euler, ideal MHD and GLM contributions,
\begin{equation}
\blocktensor{f}^a(\mathbf{u}) = \blocktensor{f}^{a,\supEuler} +\blocktensor{f}^{a,\supMHD}+\blocktensor{f}^{a,\supGLM}=
\begin{pmatrix} 
\rho \vec{v} \\[0.15cm]
\rho (\vec{v}\, \vec{v}^{\,T}) + p\threeMatrix{I} \\[0.15cm]
\vec{v}\left(\frac{1}{2}\rho \left\|\vec{v}\right\|^2 + \frac{\gamma p}{\gamma -1}\right)  \\[0.15cm]
\threeMatrix{0}\\ \vec{0}\\[0.15cm]
\end{pmatrix} +
\begin{pmatrix} 
\vec{0} \\[0.15cm]
\frac{1}{2 \mu_0} \|\vec{B}\|^2 \threeMatrix{I} - \frac{1}{\mu_0} \vec{B} \vec{B}^T \\[0.15cm]
\frac{1}{\mu_0} \left( \vec{v}\,\|\vec{B}\|^2 - \vec{B}\left(\vec{v}\cdot\vec{B}\right) \right) \\[0.15cm]
\vec{v}\,\vec{B}^T - \vec{B}\,\vec{v}^{\,T} \\ \vec{0}\\[0.15cm]
\end{pmatrix} +
\begin{pmatrix} 
\vec{0} \\[0.15cm]
\threeMatrix{0} \\[0.15cm]
\frac{c_h}{\mu_0} \psi \vec{B} \\[0.15cm]
c_h \psi \threeMatrix{I} \\ c_h \vec{B}\\[0.15cm]
\end{pmatrix},
\label{eq:advective_fluxes}
\end{equation}
where $p$ is the gas pressure, $\threeMatrix{I}$ is the $3\times 3$ identity matrix, $\mu_0$ is the permeability of the medium, and $c_h$ is the \textit{hyperbolic divergence cleaning speed}.

We close the system with the (GLM) calorically perfect gas assumption \cite{Derigs2018},
\begin{equation}
p = (\gamma-1)\left(\rho  E - \frac{1}{2}\rho\left\|\vec{v}\right\|^2 - \frac{1}{2 \mu_0}\|\vec{B}\|^2 - \frac{1}{2 \mu_0}\psi^2\right),
\label{eqofstate}
\end{equation}
where $\gamma$ denotes the heat capacity ratio, which we take as $\gamma=5/3$ for all examples of this section.

The non-conservative term has two main components, $\noncon = \noncon^\supMHD + \noncon^\supGLM$, with
\begin{align}
\noncon^\supMHD &= (\Nabla \cdot \vec{B}) \phiMHD =  \left(\Nabla \cdot \vec{B}\right) 
\left( 0 \,,\, \mu_0^{-1} \vec{B} \,,\, \mu_0^{-1} \vec{v}\cdot\vec{B} \,,\,  \vec{v} \,    ,\, 0 \right)^T \,, \label{Powell}\\ 
\noncon^\supGLM &= \phiGLM \cdot \Nabla \psi \quad =   \phiGLMs_1 \,\frac{\partial \psi}{\partial x} + \phiGLMs_2 \frac{\partial \psi}{\partial y} + \phiGLMs_3 \frac{\partial \psi}{\partial z} \,,\label{NC_GLM}
\end{align}
where $\phiGLM$ is a block vector with
\begin{equation}\label{Galilean}
\phiGLMs_\ell = \mu_0^{-1} \left(0 \,,\, 0\,,\,0\,,\,0\,,\,  v_\ell \psi \,,\, 0\,,\,0\,,\,0\,,\, v_\ell \right)^T, \quad \ell = 1,2,3.
\end{equation} 
The first non-conservative term, $\noncon^\supMHD$, is the well-known Godunov-Powell ``source'' term \cite{Powell2001}, which is needed to retain entropy stability \cite{godunov1972symmetric,powell1994approximate,powell1999solution,Derigs2018,Chandrashekar2016} and positivity of the low-order method \cite{wu2018positivity,wu2019provably}.
The second non-conservative term, $\noncon^\supGLM$, is required to retain Galilean invariance \cite{Munz2000,Dedner2002,Derigs2017,Derigs2018}.

We note that for a magnetic field with vanishing divergence, $\Nabla \cdot \vec{B} = 0$, equation \eqref{eq:GLM-MHD} reduces to the conservative ideal MHD equations, which describe the conservation of mass, momentum, energy, and magnetic flux.

\subsection{Discretization}

We discretize the GLM-MHD system using two different high-order SBP schemes,
\begin{enumerate}[(i)]
    \item a fourth-order accurate DGSEM on Gauss-Lobatto nodes \cite{Bohm2018}, which uses four degrees of freedom in each direction of each element,
    \item and a fourth-order accurate finite difference method with $13$ equispaced nodes in each direction of each element \cite[Appendix A]{fernandez2014review},
\end{enumerate}
and combine them with a first-order finite volume scheme at the element-wise (see, e.g., \cite{Rueda-Ramirez2020}) and subcell-wise levels.

The fourth-order accurate finite difference method \cite[Appendix A]{fernandez2014review} needs at least $13$ equispaced nodes in each direction of each element.
We chose to use the minimum number of nodes, such that the dissipation is applied as locally as possible by the element-wise limiting method.

We use the standard Rusanov flux for the surface numerical fluxes and the entropy-conservative flux of \citet{hindenlang2019new} for the volume numerical two-point fluxes.
This volume numerical two-point flux is constructed to mimic the curl in the induction equation and, hence, the volume term does not contribute to the divergence error of the method.

The Powell numerical non-conservative two-point term can be written in many forms that are algebraically equivalent when introduced in the SBP discretization. 
In this work, we use the expression developed by \citet{rueda2022entropy}, as it is a product of local and symmetric parts,
\begin{empheq}{align} \label{eq:numNonCons3D_other}
\numnonconsSxi{\tilde{\Jan}}_{(i,m)jk} := 
\underbrace{
 \phiMHD_{ijk}
 }_{\Jan^{\mathrm{loc}}_{ijk}} 
 \underbrace{
 \avg{\vec{B}}_{(i,m)jk} \cdot \avg{J\vec{a}^1}_{(i,m)jk}
 }_{\Jan^{\mathrm{sym}}_{(i,m)jk}},
\end{empheq}
where $\avg{\cdot}$ denotes the average operator, $J$ is the Jacobian of the element mapping and $\vec{a}$ is the contravariant metric vector (see Appendix~\ref{app:3D} for more details).
We remark that \eqref{eq:numNonCons3D_other} also holds for the Brackbill and Janhunen non-conservative terms.

Finally, the GLM non-conservative two-point term is written as \cite{rueda2022entropy}
\begin{empheq}{align} \label{eq:numNonCons3D_other2}
\numnonconsSxi{\tilde{\Jan}}_{(i,m)jk} :=
\underbrace{
\phiGLM_{ijk} \cdot J\vec{a}^1_{ijk}
 }_{\Jan^{\mathrm{loc}}_{ijk}} 
\underbrace{
\avg{\psi}_{(i,m)jk}
}_{\Jan^{\mathrm{sym}}_{(i,m)jk}},
\end{empheq}
which is also the product of a local and a symmetric non-conservative term.

To advance in time, we use the explicit third-order three-stages strong-stability-preserving Runge-Kutta method of \citet{shu1988efficient}.
The CFL restrictions of the DGSEM and SBP finite difference methods are similar.
However, since there are different algebraic expressions for the CFL condition of DGSEM and SBP finite difference methods in the literature, we use a constant stable time-step size throughout the simulations presented below, to allow a focus on the spatial discretization behavior.

\subsection{Orszag-Tang Vortex Problem}

The Orszag-Tang vortex problem \cite{orszag1979small} is a 2D inviscid MHD problem that is widely used to test the robustness of MHD codes \cite{Ciuca2020,Chandrashekar2016,Derigs2018}.
A smooth initial condition evolves into complex shock-shock interactions and transitions into supersonic MHD turbulence.

We tessellate the simulation domain, $\Omega = [0,1]^2$, with a Cartesian grid of $79 \times 79$ elements for the SBP-FD scheme, leading to a mesh with $1027 \times 1027$ degrees of freedom, and with $256 \times 256$ elements for the DGSEM scheme, leading to a mesh with $1024 \times 1024$ degrees of freedom.
Furthermore, we use periodic boundary conditions and set the initial condition to \cite{Ciuca2020,Chandrashekar2016}
\begin{align*}
\rho(x,y,t=0) &= \frac{25}{36 \pi},
&  p(x,y,t=0) &= \frac{5}{12 \pi}, \\
 v_1(x,y,t=0) &= - \sin (2 \pi y), 
&v_2(x,y,t=0) &=   \sin (2 \pi x), \\
 B_1(x,y,t=0) &= -\frac{1}{\sqrt{4 \pi}} \sin (2 \pi y), 
&B_2(x,y,t=0) &= -\frac{1}{\sqrt{4 \pi}} \sin (4 \pi x),\\
 \psi(x,y,t=0) &= 0.
\end{align*}

We run the simulation until the final time $t=1$ and use a constant time-step size $\Delta t = 8 \times 10^{-5}$, which is stable for both discretization schemes.
For this problem, we use a constant hyperbolic divergence cleaning speed $c_h=1$, which does not decrease the maximum allowable time-step size.

To deal with discontinuities and provide physically relevant solutions, we compare two different techniques to obtain the blending coefficient $\alpha$ for the hybrid SBP/FV methods: an \textit{a-posteriori} technique based on flux-corrected transport \cite{zalesak1979fully,kuzmin2010failsafe,kuzmin2020monolithic} and invariant domain preserving \cite{guermond2019invariant,maier2021efficient,Pazner2020} techniques, and an \textit{a-priori} method that uses a feature-based indicator.

\subsubsection{\textit{A-posteriori} Limiting}

We impose a local discrete maximum principle type constraint on the density at each RK stage and each degree of freedom $ij$,
\begin{equation} \label{eq:IDPconditionOT}
    \min_{k \in \NN (ij)} {\rho}^{\FV}_{k}
    \le \rho_{ij} \le
    \max_{k \in \NN (ij)} {\rho}^{\FV}_{k},
\end{equation}
where $\NN (ij)$ is the low-order stencil at node $ij$, $\rho_k^{\FV}$ is the solution at the next RK stage that is obtained with the low-order method for node $k$.
If limiting is needed to enforce \eqref{eq:IDPconditionOT}, we further impose a discrete minimum principle on a modified specific entropy \cite{guermond2019invariant,maier2021efficient},
\begin{equation} \label{eq:IDPconditionEnt}
    \min_{k \in \NN (ij)} \theta({\state{u}}^{\FV}_{k})
    \le \theta({\state{u}}_{ij}),
\end{equation}
where the use of the modified specific entropy, $\theta = e \rho^{1- \gamma}$, guarantees the fulfillment of a discrete (global) entropy inequality and is a particularly efficient choice, as $\theta$ is computationally cheaper to evaluate than the \textit{standard} specific entropy, $s=\ln (p \rho^{-\gamma})$.

This kind of limiting strategy is called \textit{a-posteriori} limiting because conditions \eqref{eq:IDPconditionOT} and \eqref{eq:IDPconditionEnt} require that both the low- and high-order operators are evaluated, and the bounds-preserving correction is done after each RK step is taken for the schemes.

To impose \eqref{eq:IDPconditionOT}, we compute $\alpha_{ij}$ at each degree of freedom of each element using a Zalesak-type limiter \cite{zalesak1979fully,RUEDARAMIREZ2022}.
To impose \eqref{eq:IDPconditionEnt} we perform a line search with a Newton-bisection method to obtain the necessary $\alpha_{ij}$ at each degree of freedom of each element, as explained in \cite{RUEDARAMIREZ2022,Rueda-Ramirez2021}.
Note that the modified specific entropy is better suited for line searches with a Newton's method than the \textit{standard} specific entropy \cite{guermond2019invariant,maier2021efficient}.

For the element-wise limiting strategy, \eqref{eq:Elem_blend}, we choose the blending coefficient of each element simply as the maximum of all degrees of freedom of the element,
\begin{equation}
    \alpha = \max_{ij} \alpha_{ij}.
\end{equation}
For the subcell-wise limiting strategy, \eqref{eq:Subcell_blend}-\eqref{eq:Subcell_blend2}, we choose the blending coefficient of each interface between two nodes as the maximum over the selected coefficient for the two nodes, e.g.,
\begin{equation}
    \alpha_{(i,i+j)j} = \max (\alpha_{ij},\alpha_{i+1,j}).
\end{equation}

Conditions \eqref{eq:IDPconditionOT} and \eqref{eq:IDPconditionEnt} are particularly restrictive and can affect the order of accuracy of the scheme significantly \cite{guermond2019invariant,maier2021efficient,kuzmin2020monolithic,Pazner2020,RUEDARAMIREZ2022}.
Therefore, a common practice is to use a smoothness indicator and enforce conditions \eqref{eq:IDPconditionOT} and \eqref{eq:IDPconditionEnt} only where the solution is not smooth.
The smoothness detection can be done using modal troubled-cell indicators \cite{Persson2006,RUEDARAMIREZ2022}, comparing a functional of the solution within the scheme's stencil \cite{guermond2019invariant}, computing higher-order derivatives of the solution using reconstructions \cite{diot2012improved}, among other techniques.
Since all these smoothness indicator techniques depend on the node spacing and element sizes, different methodologies, such as for instance DG and FD, will flag different elements. Hence, we do not use a smoothness indicator to pre-select the elements for limiting in this section to provide a direct comparison of the limiting performance of the DGSEM and finite difference methods.

\begin{figure}[h!]
	\centering
	\includegraphics[trim=385 825 1819 0 ,clip,width=0.26\linewidth]{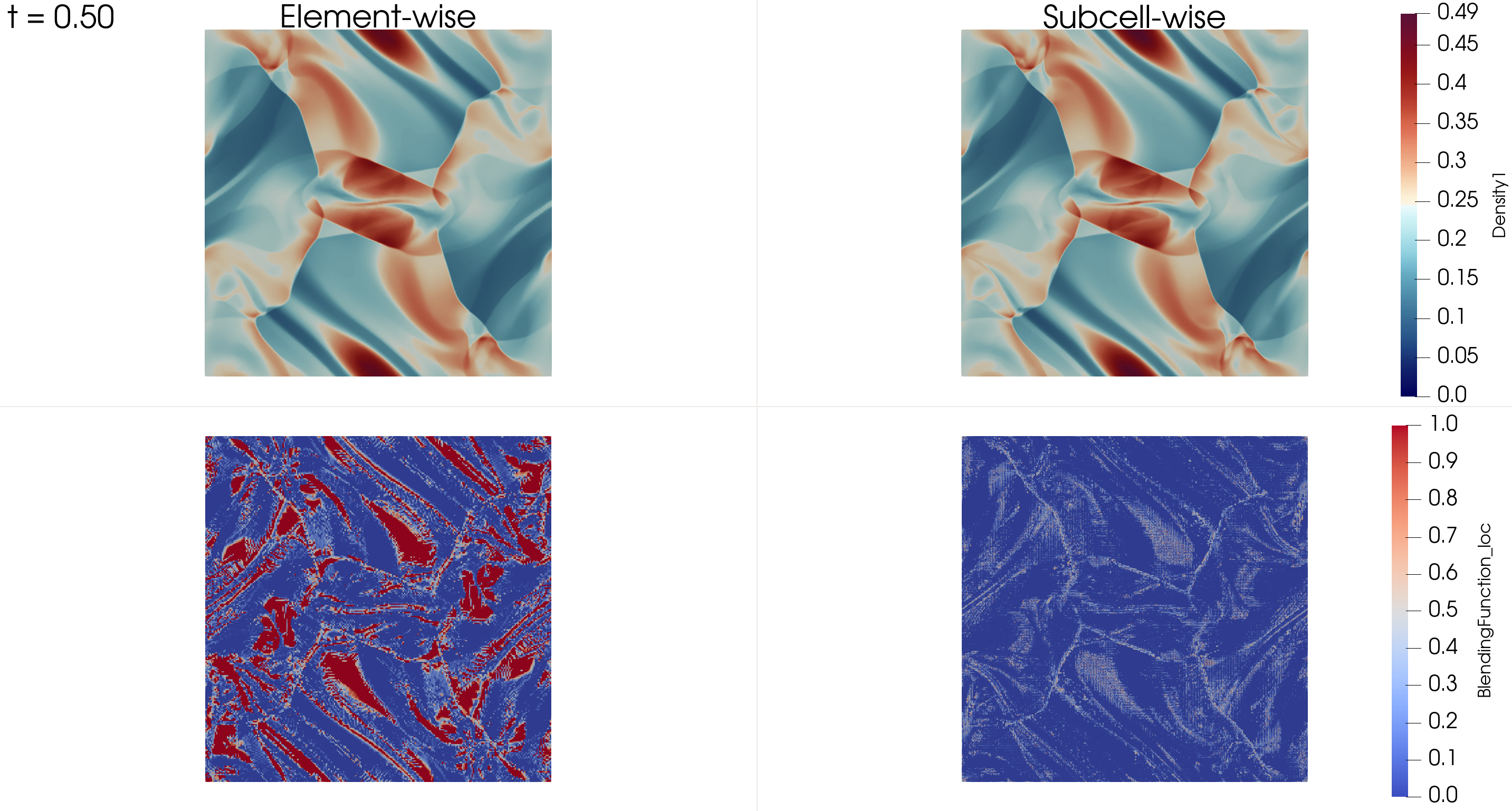}
	\includegraphics[trim=1819 825 385 0 ,clip,width=0.26\linewidth]{figs/OT_IDP/OT_t_0.5_DGSEM.png}
	\includegraphics[trim=2780 800 0 0 ,clip,height=0.26\linewidth]{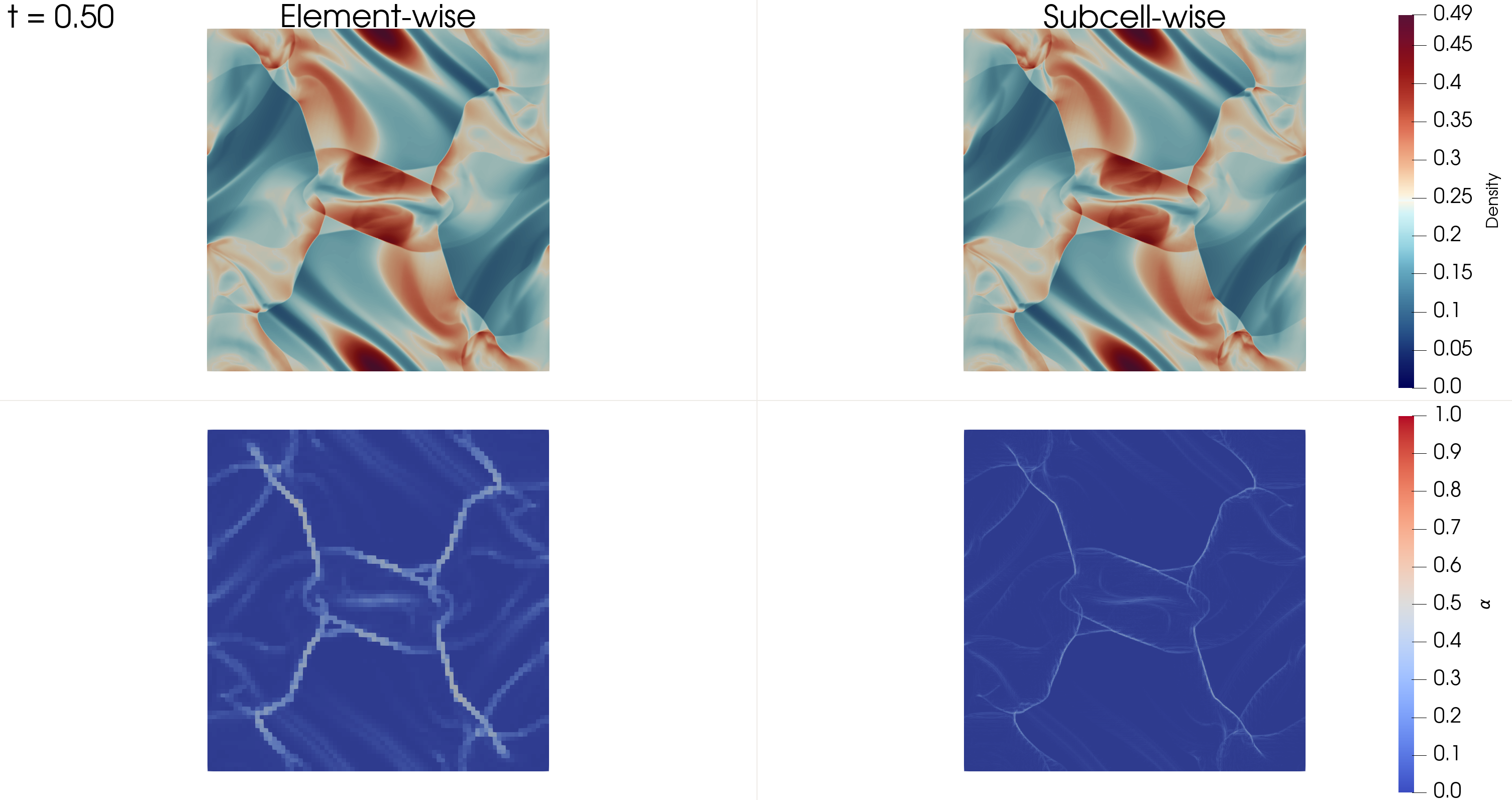}
	
	\includegraphics[trim=385 55 1819 825 ,clip,width=0.26\linewidth]{figs/OT_IDP/OT_t_0.5_DGSEM.png}
	\includegraphics[trim=1819 55 385 825 ,clip,width=0.26\linewidth]{figs/OT_IDP/OT_t_0.5_DGSEM.png}
	\includegraphics[trim=2780 0 0 800 ,clip,height=0.26\linewidth]{figs/OT_Loehner/OT_t_0.5_FD.png}
	\caption{Density and blending coefficient $\alpha$ for the Orszag-Tang vortex problem obtained with the hybrid DGSEM/FV method at $t=0.50$ using \textit{a-posteriori} element-wise blending (left) and subcell-wise blending (right) based on \eqref{eq:IDPconditionOT} and \eqref{eq:IDPconditionEnt}.}
	\label{fig:OT_IDP/OT_DG}
	
	\includegraphics[trim=385 825 1819 0 ,clip,width=0.26\linewidth]{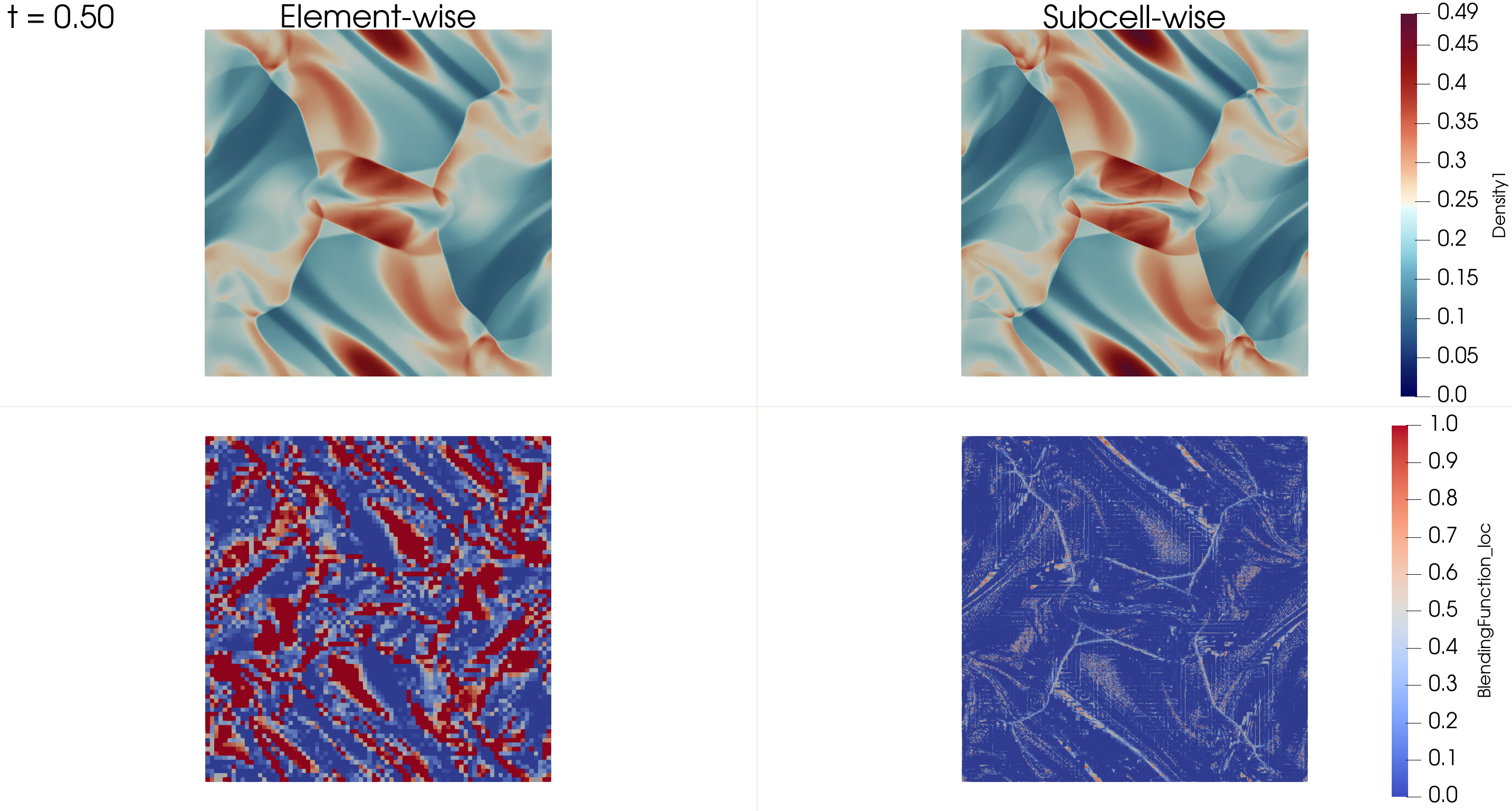}
	\includegraphics[trim=1819 825 385 0 ,clip,width=0.26\linewidth]{figs/OT_IDP/OT_t_0.5_FD.png}
	\includegraphics[trim=2780 800 0 0 ,clip,height=0.26\linewidth]{figs/OT_Loehner/OT_t_0.5_FD.png}
	
	\includegraphics[trim=385 55 1819 825 ,clip,width=0.26\linewidth]{figs/OT_IDP/OT_t_0.5_FD.png}
	\includegraphics[trim=1819 55 385 825 ,clip,width=0.26\linewidth]{figs/OT_IDP/OT_t_0.5_FD.png}
	\includegraphics[trim=2780 0 0 800 ,clip,height=0.26\linewidth]{figs/OT_Loehner/OT_t_0.5_FD.png}
	\caption{Density and blending coefficient $\alpha$ for the Orszag-Tang vortex problem obtained with the hybrid FD/FV method at $t=0.50$ using \textit{a-posteriori} element-wise blending (left) and subcell-wise blending (right) based on \eqref{eq:IDPconditionOT} and \eqref{eq:IDPconditionEnt}.}
	\label{fig:OT_IDP/OT_FD}
\end{figure}

Figures \ref{fig:OT_IDP/OT_DG} and \ref{fig:OT_IDP/OT_FD} show the density and blending coefficient contours at $t=0.5$ obtained with the element-wise and subcell-wise hybrid SBP/FV methods for the \textit{a-posteriori} limiting.
Subcell-wise limiting clearly improves the solution quality with respect to element-wise limiting for both DGSEM and high-order SBP-FD schemes.
Both methods perform similarly for subcell-wise limiting, but the element-wise high-order SBP-FD/FV scheme is more dissipative than the element-wise DGSEM/FV scheme since its elements are larger by a factor of approximately $16$ (in area).
Thus, the dissipation of the first-order FV is applied over a larger area when conditions \eqref{eq:IDPconditionOT} or \eqref{eq:IDPconditionEnt} are violated.
If we used more nodes in each direction of each element, as is common for high-order FD methods, the element-wise hybrid scheme would produce more dissipative results, with the extreme being a one-element (block) FD approach being dissipated at every node in every time step.

To compare the dissipation properties of the methods quantitatively, we compute the evolution of the mean blending coefficient and the total (mathematical) entropy in the domain.

The mean blending coefficient is computed as
\begin{equation} \label{eq:meanAlpha}
\bar \alpha (t) = 
\frac{1}{n_r} \sum_{r=1}^{n_r} \left( \frac{1}{V} \int_{\Omega} \alpha(t_r) \d \vec{x} \right)
\approx \frac{1}{n_r} \sum_{r=1}^{n_r} \left( \frac{1}{V} \sum_{e=1}^K \sum_{i,j=0}^N m_{ij} (\alpha^e_{ij})^r \right),
\end{equation}
where $e \in [1,K]$ denotes the element index, $K$ is the number of elements in the domain, $i,j \in [0,N]$ are the node indexes, $N$ is the polynomial degree, $(\alpha^e_{ij})^r$ is the blending coefficient of node $ij$ of element $e$ at the RK stage $r$, $n_r$ is the number of RK stages taken from $t-\Delta \tau$ to $t$, $\Delta \tau$ is the sample time, and $V$ is the area of the domain. A value $\bar \alpha = 1$ means that the entire domain uses a first-order FV method, whereas $\bar \alpha = 0$ means that it uses only the high-order SBP method.

The total entropy is computed as
\begin{equation} \label{eq:totalEntropy}
S_{\Omega}(t) = 
- \int_{\Omega}  \frac{\rho s}{\gamma-1} \d \vec{x}
\approx - \sum_{e=1}^K \sum_{i,j=0}^N m_{ij} \frac{\rho^e_{ij} s^e_{ij}}{\gamma-1},
\end{equation}
where $s = \ln\left(p \rho^{-\gamma}\right)$ is the specific entropy.

Figure \ref{fig:OT_IDP/OT_alpha_ent} shows the evolution of the mean blending coefficient $\bar \alpha$ (with a sample time $\Delta \tau = 0.01$) and the total entropy $S_{\Omega}$ in time for the four different schemes analyzed in this section.
The figure shows that both subcell-wise schemes add a similar amount of low-order stabilization, which results in a similar entropy dissipation.
The amount of limiting added by the element-wise schemes is higher throughout the simulation, as observed in Figures \ref{fig:OT_IDP/OT_DG} and \ref{fig:OT_IDP/OT_FD}, which results in a higher entropy dissipation.
As discussed before, the element-wise high-order SBP-FD/FV scheme is more dissipative than the element-wise
DGSEM/FV scheme since its elements are larger.

Finally, Figure~\ref{fig:OT_IDP/OT_slice} shows the pressure along a slice of the simulation domain for the four different \textit{a-posteriori} methods and a comparison with the results obtained by the second-order solver of the astrophysics code Athena \cite{Stone2008} on a mesh with $2048^2$ degrees of freedom.
As expected, the schemes with a higher dissipation are unable to capture small features of the solution.

\begin{figure}
    \centering
	\includegraphics[trim=0 0 0 0 ,clip,width=0.45\linewidth]{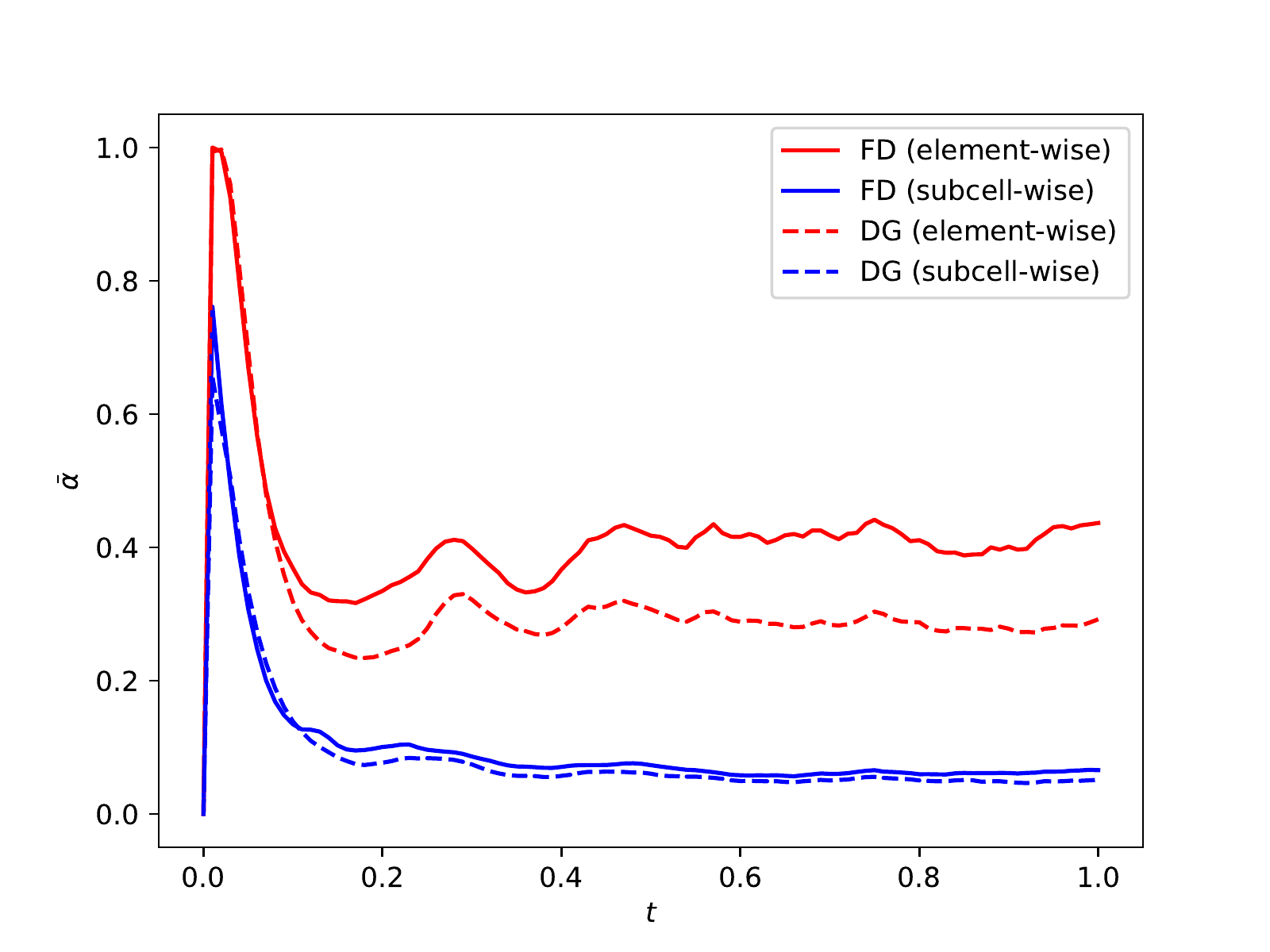}
	\includegraphics[trim=0 0 0 0 ,clip,width=0.45\linewidth]{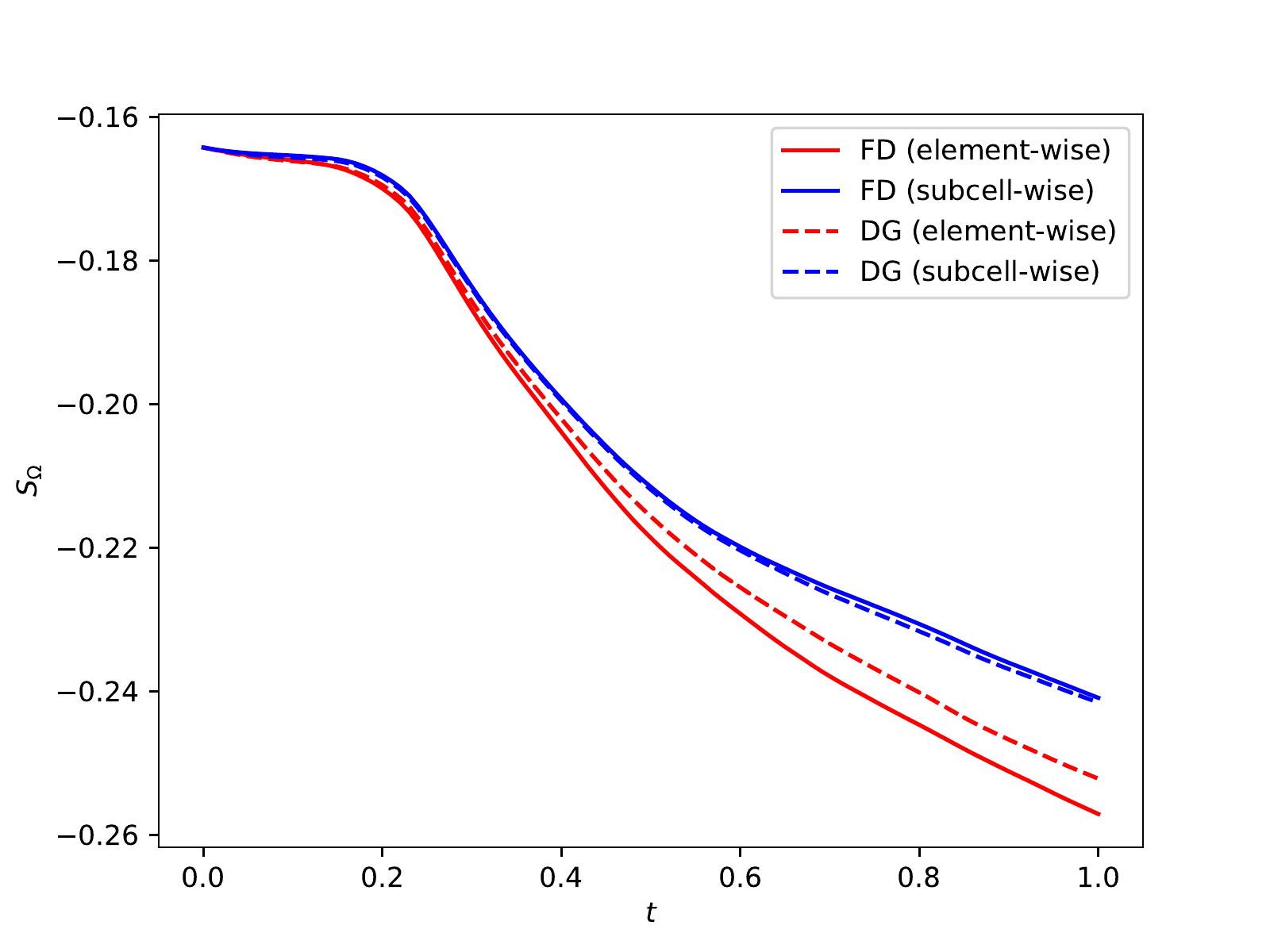}
	\caption{Evolution of the mean blending coefficient and the total entropy for the Orszag-Tang vortex test using \textit{a-posteriori} limiting based on \eqref{eq:IDPconditionOT} and \eqref{eq:IDPconditionEnt}.}
	\label{fig:OT_IDP/OT_alpha_ent}

    \centering
	\includegraphics[trim=0 0 0 0 ,clip,width=0.8\linewidth]{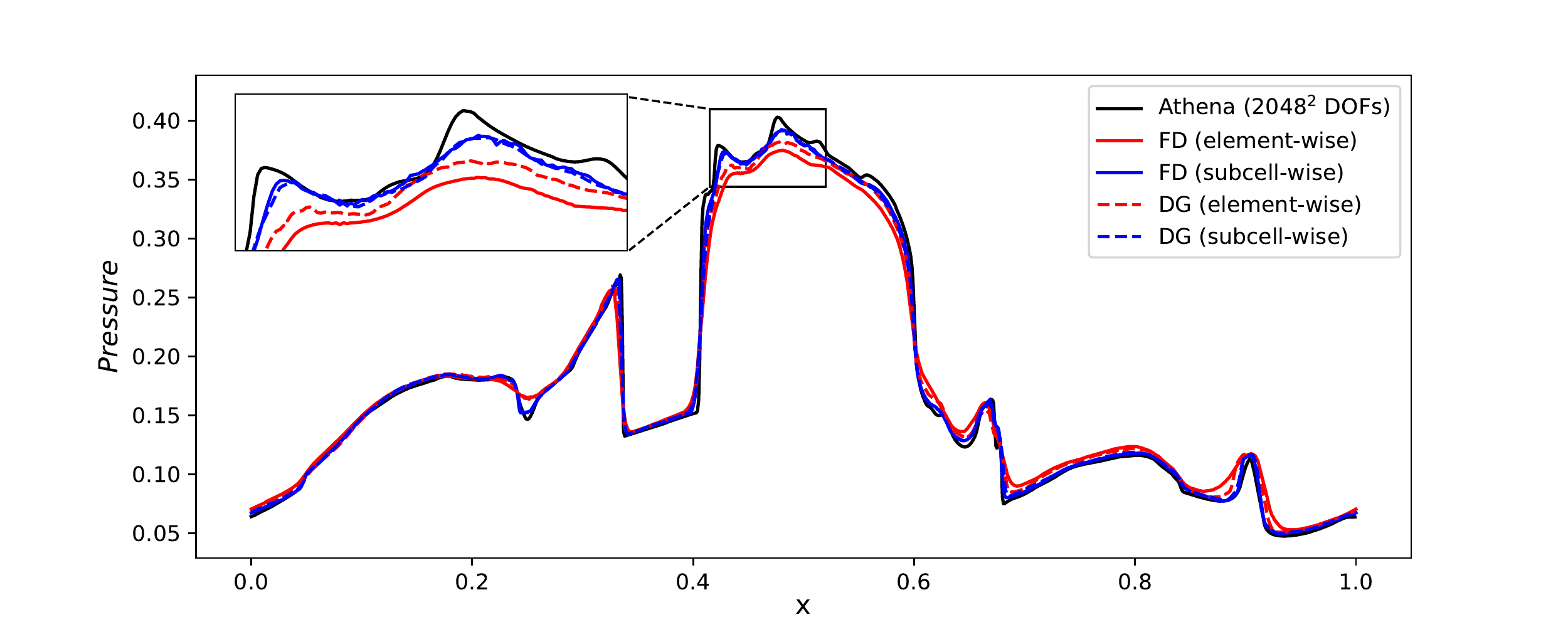}
	\caption{Slice of dimensionless pressure for the Orszag-Tang vortex test at $t=0.5$ and y=0.4277 using \textit{a-posteriori} limiting based on \eqref{eq:IDPconditionOT} and \eqref{eq:IDPconditionEnt}, and comparison with the second-order solver of Athena \cite{Stone2008} on a finer mesh \cite{Rueda-Ramirez2020}.}
	\label{fig:OT_IDP/OT_slice}
\end{figure}

\subsubsection{\textit{A-priori} Limiting}

To compute the blending coefficient before the evaluation of the SBP and FV operators (\textit{a priori}), we use a modification of the indicator of \citet{lohner1987adaptive}, which computes a modified second derivative, normalized by the gradient.
In particular, we evaluate the maximum over each spatial direction,
\begin{equation} \label{eq:Loehner}
    \alpha_{ij} \approx \max \left\lbrace \max_{d=1,2} \left(
    \frac
    {\left| \bigpartialderiv{{}^2u_{ij}}{x^2_d}\right|}
    {2\left| \bigpartialderiv{u_{ij}}{x_d}\right| + \epsilon \bigpartialderiv{{}^2|u|_{ij}}{x^2_d} }
    \right), 1 \right\rbrace,
\end{equation}
where we use the indicator quantity $u=\rho \, p$.
The last term in the denominator of \eqref{eq:Loehner} acts as a filter that prevents that the indicator is triggered in small ripples, with the heuristically determined constant $\epsilon=0.2$.

We base our implementation of the Löhner indicator on the Flash code\footnote{\url{https://flash.rochester.edu/site/flashcode/user_support/flash4_ug/node14.html\#SECTION05163000000000000000}} \cite{fryxell2000flash}, where the partial derivative terms in \eqref{eq:Loehner} are approximated with central finite differences.
However, we modify the expression to take into account the possibly irregular grids of high-order SBP operators,
\begin{equation} \label{eq:Loehner_FD}
    \frac
    {\left| \bigpartialderiv{{}^2u_{i}}{x^2}\right|}
    {2\left| \bigpartialderiv{u_{i}}{x}\right| + \epsilon \bigpartialderiv{{}^2|u|_{i}}{x^2} }
    \approx
    \frac
    {\left| \Delta x^-_{i} u_{i+1} - (\Delta x^-_{i}+ \Delta x^+_{i}) u_{i} + \Delta x^+_{i} u_{i-1}  \right|}
    {\Delta x^-_{i}|u_{i+1}-u_{i}| + \Delta x^+_{i} |u_{i}-u_{i-1}| 
    + \epsilon \left( \Delta x^-_{i} |u_{i+1}| + (\Delta x^-_{i}+ \Delta x^+_{i}) |u_{i}| + \Delta x^+_{i} |u_{i-1}|  \right) },
\end{equation}
where $\Delta x^-_{i}:=x_i-x_{i-1}$ and $\Delta x^+_{i}\coloneqq x_{i+1} - x_i$ denote grid size to the left and right of node $i$.
When $\Delta x^-_{i}=\Delta x^+_{i}$, \eqref{eq:Loehner_FD} reduces to the finite difference formula used in the Flash code.
However, for the irregular LGL-DGSEM grids, we obtained better results with \eqref{eq:Loehner_FD} than with the expression that assumes $\Delta x^-_{i}=\Delta x^+_{i}$.

\begin{figure}[h!]
	\centering
	\includegraphics[trim=409 852 1915 0 ,clip,width=0.26\linewidth]{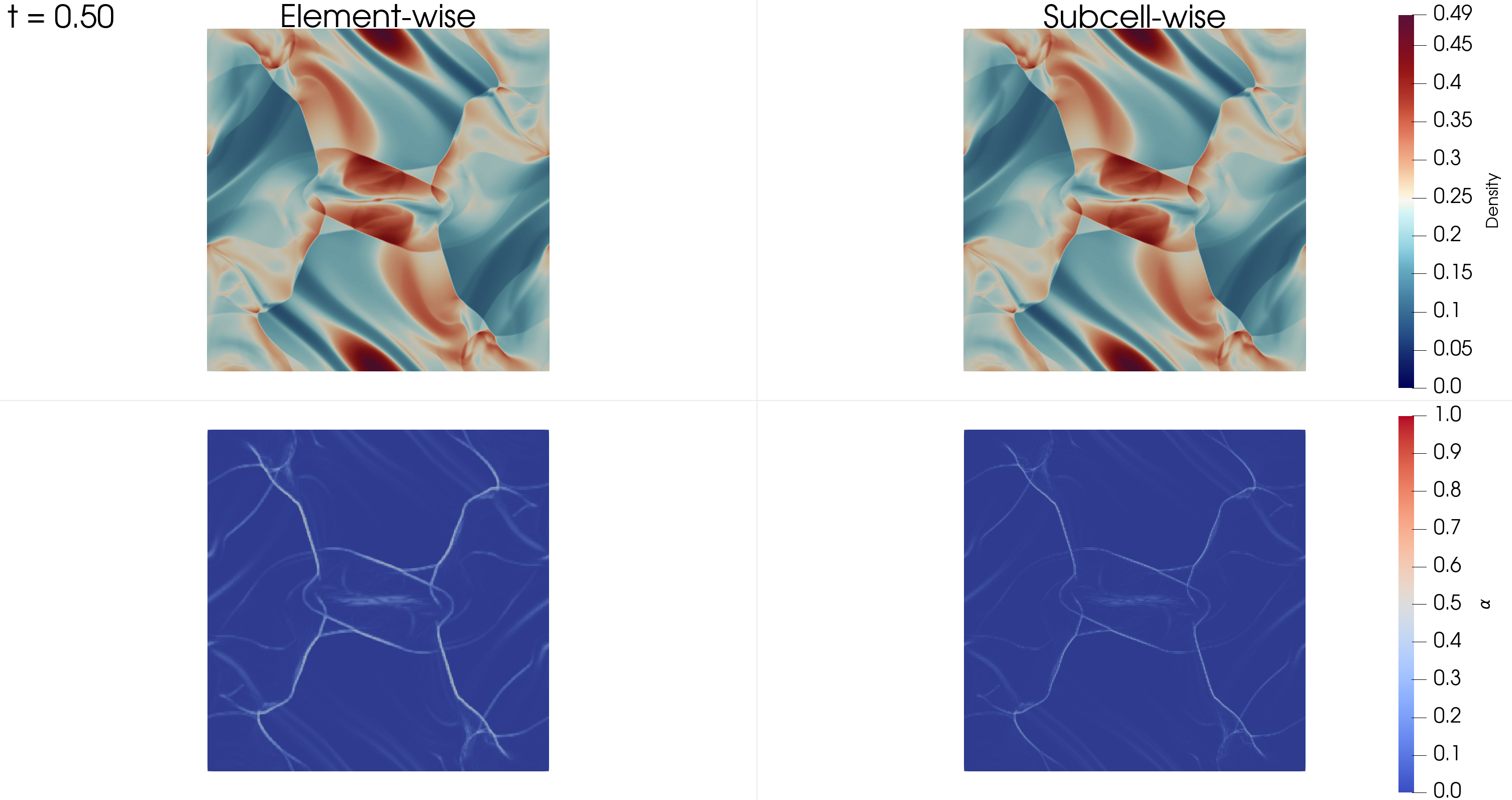}
	\includegraphics[trim=1915 852 409 0 ,clip,width=0.26\linewidth]{figs/OT_Loehner/OT_t_0.5_DGSEM.png}
	\includegraphics[trim=2780 800 0 0 ,clip,height=0.26\linewidth]{figs/OT_Loehner/OT_t_0.5_DGSEM.png}
	
	\includegraphics[trim=409 55 1915 852 ,clip,width=0.26\linewidth]{figs/OT_Loehner/OT_t_0.5_DGSEM.png}
	\includegraphics[trim=1915 55 409 852 ,clip,width=0.26\linewidth]{figs/OT_Loehner/OT_t_0.5_DGSEM.png}
	\includegraphics[trim=2780 0 0 800 ,clip,height=0.26\linewidth]{figs/OT_Loehner/OT_t_0.5_DGSEM.png}
	\caption{Density and blending coefficient $\alpha$ for the Orszag-Tang vortex problem obtained with the hybrid DGSEM/FV method at $t=0.50$ using \textit{a-priori} element-wise blending (left) and subcell-wise blending (right) based on \eqref{eq:Loehner}.}
	\label{fig:OT_Loehner/OT_DG}
	
	\includegraphics[trim=409 852 1915 0 ,clip,width=0.26\linewidth]{figs/OT_Loehner/OT_t_0.5_FD.png}
	\includegraphics[trim=1915 852 409 0 ,clip,width=0.26\linewidth]{figs/OT_Loehner/OT_t_0.5_FD.png}
	\includegraphics[trim=2780 800 0 0 ,clip,height=0.26\linewidth]{figs/OT_Loehner/OT_t_0.5_FD.png}
	
	\includegraphics[trim=409 55 1915 852 ,clip,width=0.26\linewidth]{figs/OT_Loehner/OT_t_0.5_FD.png}
	\includegraphics[trim=1915 55 409 852 ,clip,width=0.26\linewidth]{figs/OT_Loehner/OT_t_0.5_FD.png}
	\includegraphics[trim=2780 0 0 800 ,clip,height=0.26\linewidth]{figs/OT_Loehner/OT_t_0.5_FD.png}
	\caption{Density and blending coefficient $\alpha$ for the Orszag-Tang vortex problem obtained with the hybrid FD/FV method at $t=0.50$ using \textit{a-priori} element-wise blending (left) and subcell-wise blending (right) based on \eqref{eq:Loehner}.}
	\label{fig:OT_Loehner/OT_FD}
\end{figure}

We use \eqref{eq:Loehner_FD} to evaluate \eqref{eq:Loehner} locally at each degree of freedom of each element.
To avoid additional communication costs, and taking into account that the solution quantities at the element boundaries are not uniquely defined, we compute \eqref{eq:Loehner_FD} only for the internal degrees of freedom in each direction, i.e., $i \in [1, \ldots, N-1]$.
This does not affect the shock-capturing properties of the method significantly since the normal fluxes at the element boundaries are identical for the low-order FV and high-order SBP methods.

Figures~\ref{fig:OT_Loehner/OT_DG} and \ref{fig:OT_Loehner/OT_FD} show the density and blending coefficient contours at $t=0.5$ obtained with the element-wise and subcell-wise hybrid SBP/FV methods for the \textit{a-priori} limiting.
The blending coefficient contours show the same features that we observed with the \textit{a-posteriori} limiting.
However, it is hard to distinguish strong differences in the density contours due to the very local nature of the Löhner indicator.

The amount of dissipation introduced by both subcell-wise schemes is similar, but the subcell-wise DGSEM is slightly more dissipative, as observed in the evolution of $\bar \alpha$ and $S_{\Omega}$ in Figure~\ref{fig:OT_Loehner/OT_alpha_ent}.
This slightly higher dissipation is likely to be a consequence of the node spacing and the implementation of the Löhner indicator as a compact finite-difference scheme \eqref{eq:Loehner_FD}.
Figure~\ref{fig:OT_Loehner/OT_alpha_ent} shows clearly that the most dissipative scheme is again the element-wise SBP-FD method, followed by the element-wise DGSEM.
Contrary to the \textit{a-posteriori} limiting techniques based on \eqref{eq:IDPconditionOT} and \eqref{eq:IDPconditionEnt} (without smoothness indication), all schemes that use \textit{a-priori} limiting based on \eqref{eq:Loehner} introduce very little dissipation at the beginning of the simulation, when the solution is smooth and no shocks are present in the domain.

Finally, Figure~\ref{fig:OT_Loehner/OT_slice} shows the pressure along a slice of the simulation domain for the four different \textit{a-priori} methods and a comparison with the results obtained by the second-order solver of the astrophysics code Athena \cite{Stone2008} on a mesh with $2048^2$ degrees of freedom.

\begin{figure}
    \centering
	\includegraphics[trim=0 0 0 0 ,clip,width=0.45\linewidth]{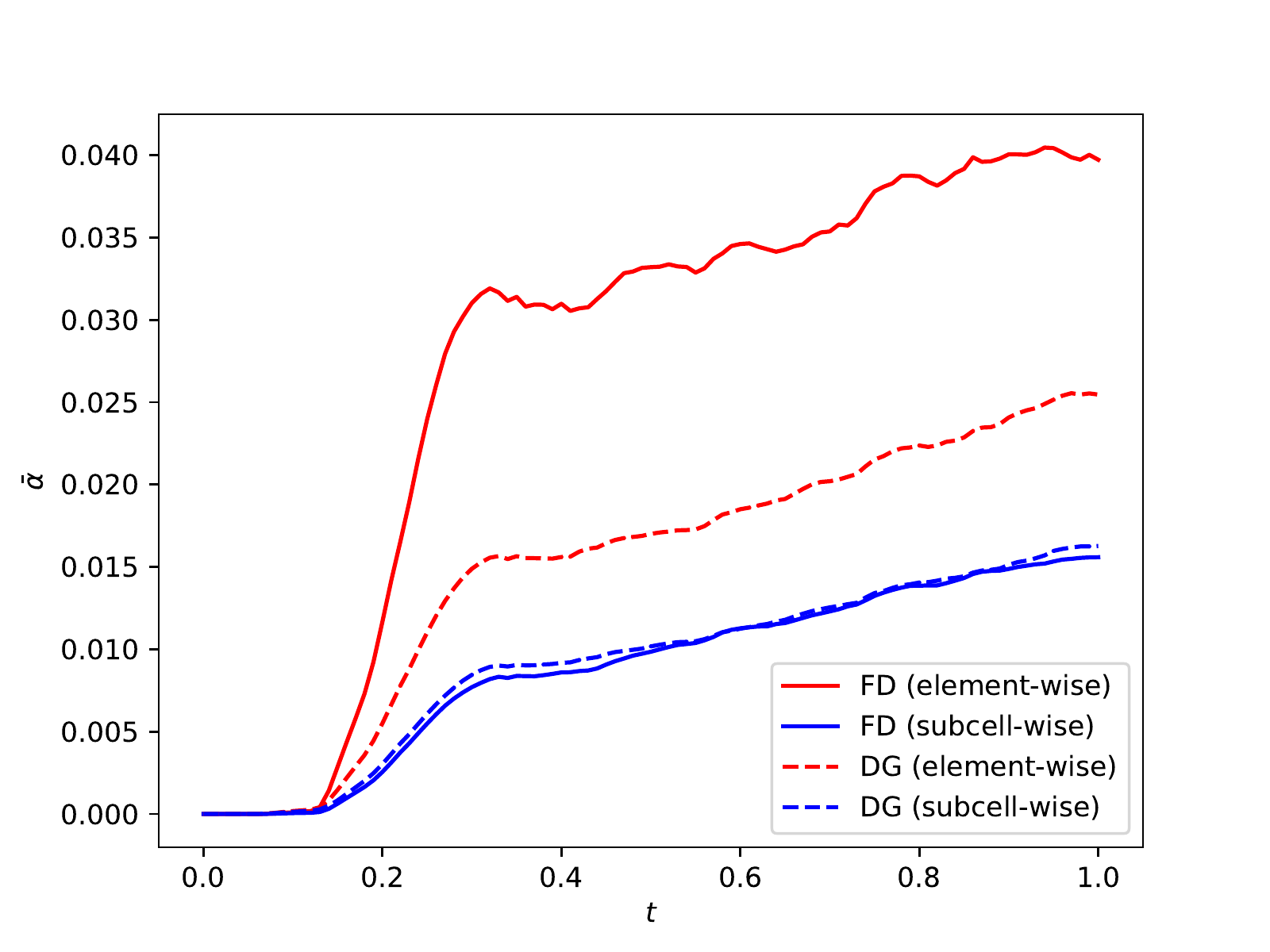}
	\includegraphics[trim=0 0 0 0 ,clip,width=0.45\linewidth]{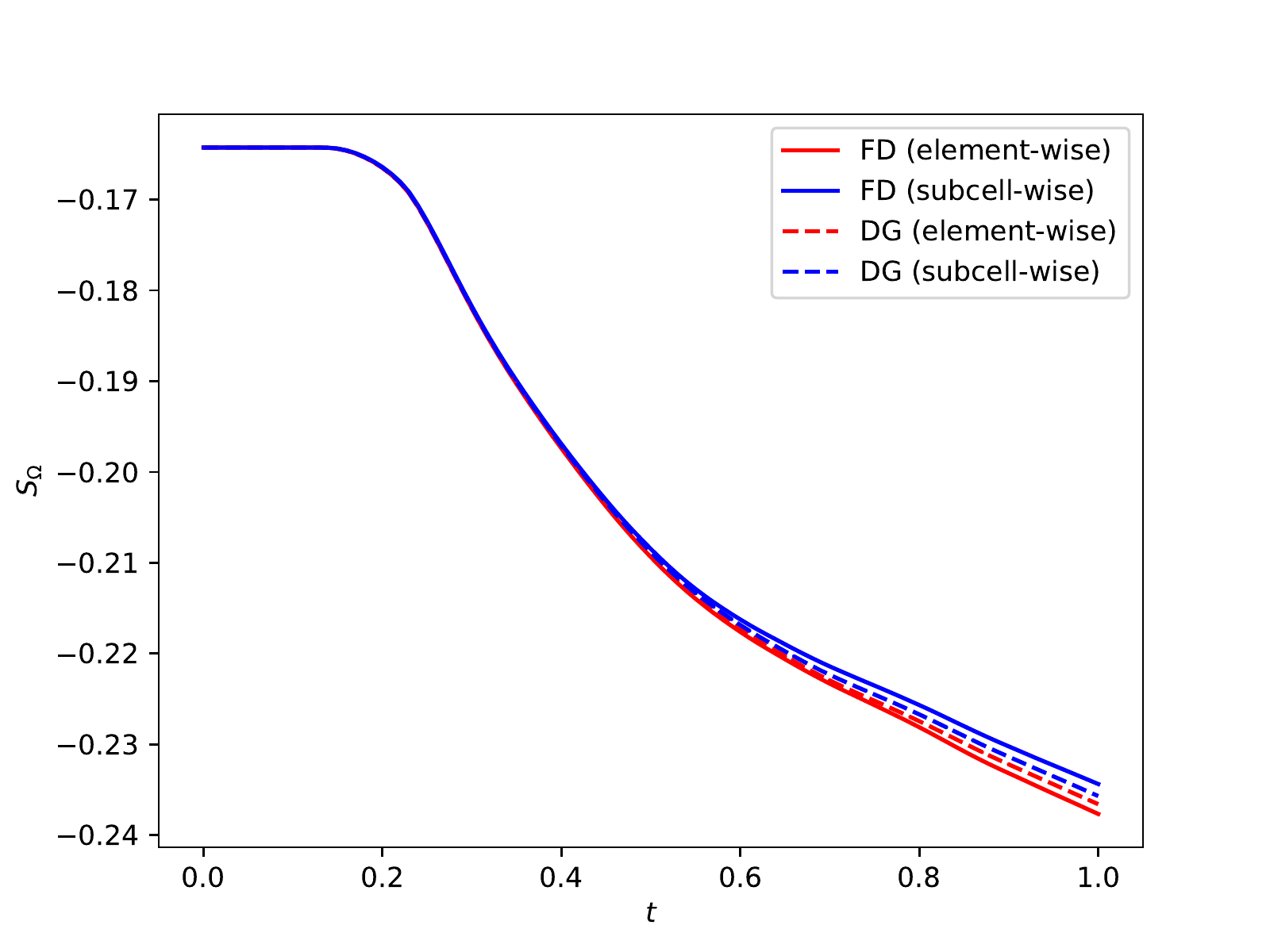}
	\caption{Evolution of the mean blending coefficient and the total entropy for the Orszag-Tang vortex test using the \textit{a-priori} limiting based on \eqref{eq:Loehner}.}
	\label{fig:OT_Loehner/OT_alpha_ent}

    \centering
	\includegraphics[trim=0 0 0 0 ,clip,width=0.8\linewidth]{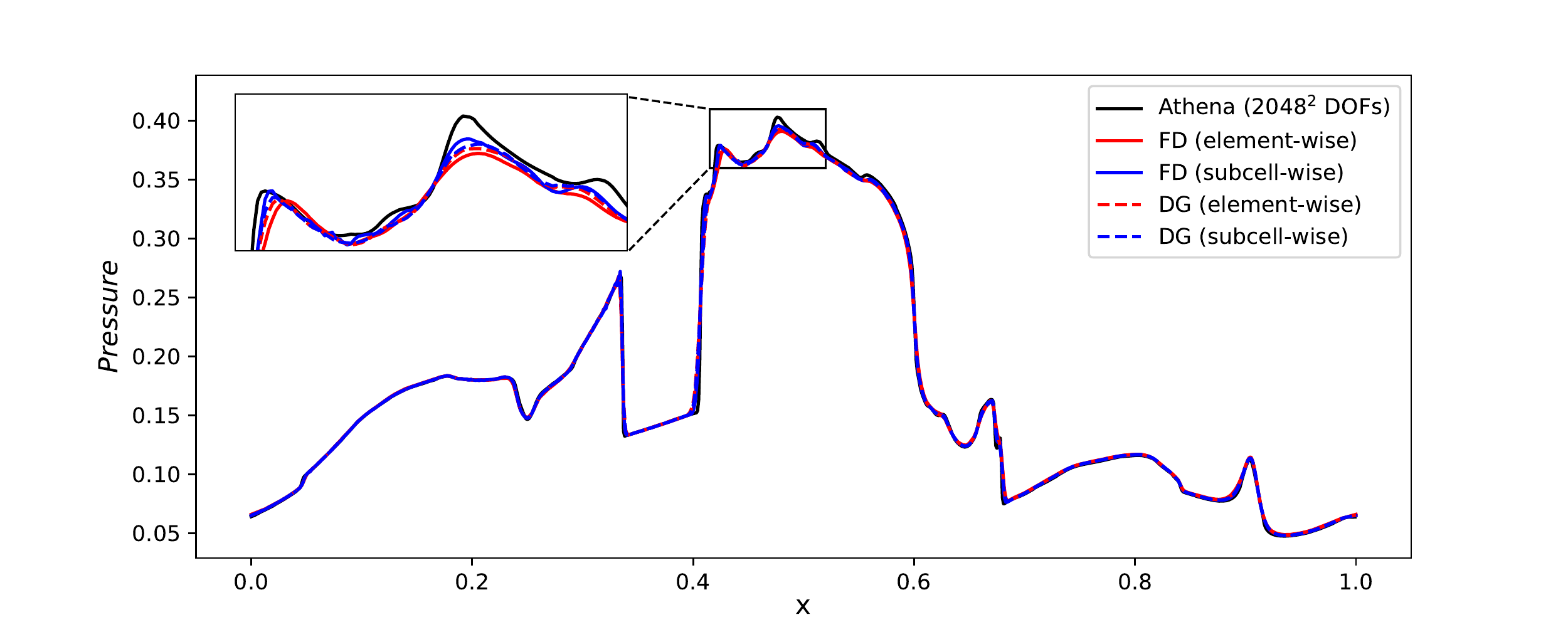}
	\caption{Slice of dimensionless pressure for the Orszag-Tang vortex test at $t=0.5$ and y=0.4277 using the \textit{a-priori} limiting, and comparison with the second-order solver of Athena \cite{Stone2008,Rueda-Ramirez2020} on a finer mesh.}
	\label{fig:OT_Loehner/OT_slice}
\end{figure}

\begin{figure}[h!]
	\centering
	\includegraphics[trim=409 852 1915 0 ,clip,width=0.26\linewidth]{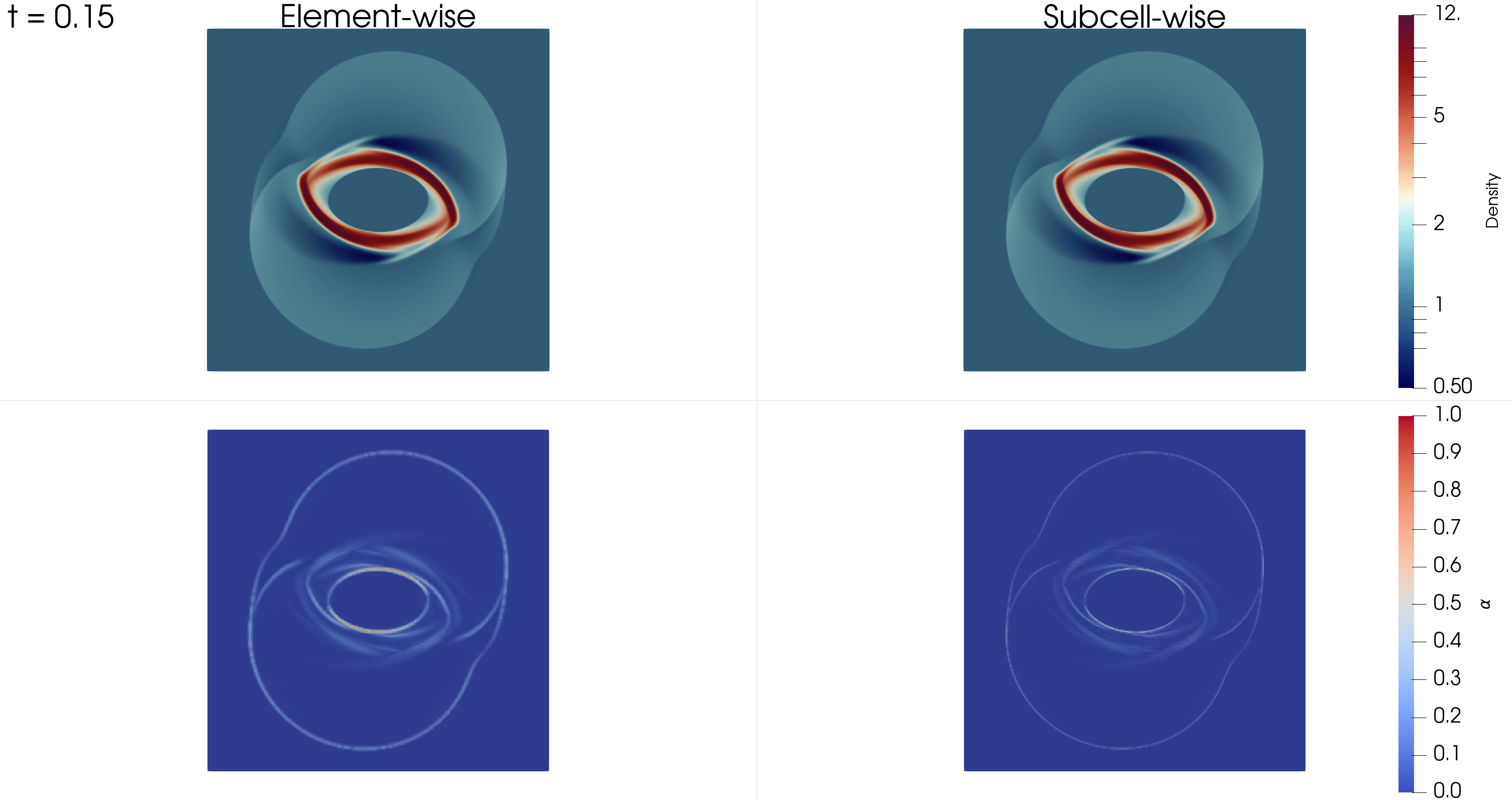}
	\includegraphics[trim=1915 852 409 0 ,clip,width=0.26\linewidth]{figs/Rotor_Loehner/Rotor_t_0.15_DGSEM.png}
	\includegraphics[trim=2780 800 0 0 ,clip,height=0.26\linewidth]{figs/Rotor_Loehner/Rotor_t_0.15_DGSEM.png}
	
	\includegraphics[trim=409 55 1915 852 ,clip,width=0.26\linewidth]{figs/Rotor_Loehner/Rotor_t_0.15_DGSEM.png}
	\includegraphics[trim=1915 55 409 852 ,clip,width=0.26\linewidth]{figs/Rotor_Loehner/Rotor_t_0.15_DGSEM.png}
	\includegraphics[trim=2780 0 0 800 ,clip,height=0.26\linewidth]{figs/Rotor_Loehner/Rotor_t_0.15_DGSEM.png}
	\caption{Density and blending coefficient $\alpha$ for the rotor problem obtained with the hybrid DGSEM/FV method at $t=0.15$ using \textit{a-priori} element-wise blending (left) and subcell-wise blending (right) based on \eqref{eq:Loehner}.}
	\label{fig:Rotor_Loehner/Rotor_DG}
	
	\includegraphics[trim=409 852 1915 0 ,clip,width=0.26\linewidth]{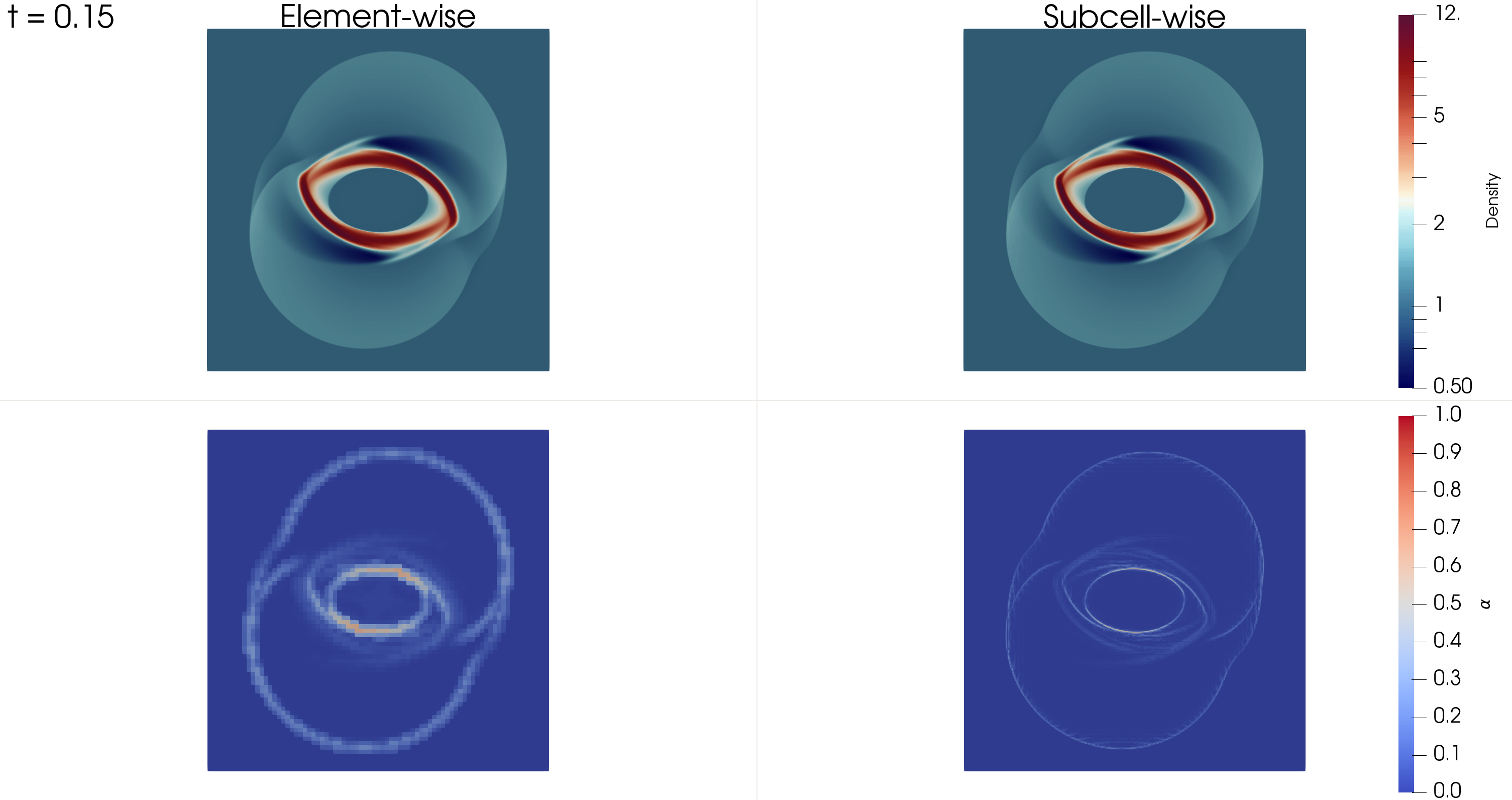}
	\includegraphics[trim=1915 852 409 0 ,clip,width=0.26\linewidth]{figs/Rotor_Loehner/Rotor_t_0.15_FD.png}
	\includegraphics[trim=2780 800 0 0 ,clip,height=0.26\linewidth]{figs/Rotor_Loehner/Rotor_t_0.15_FD.png}
	
	\includegraphics[trim=409 55 1915 852 ,clip,width=0.26\linewidth]{figs/Rotor_Loehner/Rotor_t_0.15_FD.png}
	\includegraphics[trim=1915 55 409 852 ,clip,width=0.26\linewidth]{figs/Rotor_Loehner/Rotor_t_0.15_FD.png}
	\includegraphics[trim=2780 0 0 800 ,clip,height=0.26\linewidth]{figs/Rotor_Loehner/Rotor_t_0.15_FD.png}
	\caption{Density and blending coefficient $\alpha$ for the rotor problem obtained with the hybrid FD/FV method at $t=0.15$ using \textit{a-priori} element-wise blending (left) and subcell-wise blending (right) based on \eqref{eq:Loehner}.}
	\label{fig:Rotor_Loehner/Rotor_FD}
\end{figure}

\subsection{MHD Rotor Test}

The rotor test was first proposed by \citet{balsara2009divergence} and we use the modification proposed by \citet{toth2000b}.
The simulation domain is again the unit square, $\Omega = [0,1]^2$, which we tesselate with a Cartesian grid of $79 \times 79$ elements for the SBP-FD scheme ($1027 \times 1027$ degrees of freedom), and with $256 \times 256$ elements for the DGSEM scheme ($1024 \times 1024$ degrees of freedom).

The initial condition has a constant pressure and magnetic field, and a dense rotating disk of fluid in a medium with a constant background density and velocity,
\begin{align*}
    p(x,y,t=0) &= 1,& \vec{B}(x,y,t=0) &= \left( \frac{5}{4 \pi} , 0 \right), \\
    \rho(x,y,t=0) &=
     \begin{cases}
       10 & r< r_0 \\
       1+9f & r_0 \le r < r_1 \\
       1 & r \ge r_1
     \end{cases},&
    \vec{v}(x,y,t=0) &= 
      \begin{cases}
       \frac{u_0}{r_0}\left( \frac{1}{2} - y, x - \frac{1}{2}\right), & r< r_0 \\
       \frac{f u_0}{r_0}\left( \frac{1}{2} - y, x - \frac{1}{2}\right), & r_0 \le r < r_1 \\
       \left(0,0\right), & r \ge r_1
     \end{cases},
     \\
     \psi (x,y,t=0) &= 0,
\end{align*}
with $r\coloneqq \sqrt{\left(x-\frac{1}{2}\right)^2+\left(y-\frac{1}{2}\right)^2}$, $r_0=0.1$, $r_1=0.115$, $u_0=2$, and $f=\frac{r_1-r}{r_1-r_0}$.

For this example, we use the \textit{a-priori} selection of the blending coefficient based on the Löhner indicator \eqref{eq:Loehner} with the solution quantity $u=\rho p$, as it showed to provide good results in the last section.
As before, we apply the limiting technique in an element-wise and subcell-wise manner.

Figures~\ref{fig:Rotor_Loehner/Rotor_DG} and \ref{fig:Rotor_Loehner/Rotor_FD} illustrate the density (in logarithmic scale) and blending coefficient at time $t=0.15$ for the hybrid DGSEM/FV and SBP-FD/FV schemes, respectively.
Moreover, Figure~\ref{fig:Rotor_Loehner/Rotor_alpha_ent} shows the evolution of the mean blending coefficient and the total entropy in the computational domain.
The most dissipative scheme is again the high-order SBP-FD with element-wise FV limiting (due to the element sizes), followed by the element-wise combination of DGSEM with FV, and the subcell-wise methods.
Both subcell-wise schemes show similar mean blending coefficients throughout the simulation, while the DGSEM/FV scheme uses slightly more of the FV method at the beginning of the simulation.
This additional limiting at the beginning of the simulation, which is probably caused by the node spacing of the DGSEM method, translates into a higher entropy dissipation.

\begin{figure}
    \centering
	\includegraphics[trim=0 0 0 0 ,clip,width=0.45\linewidth]{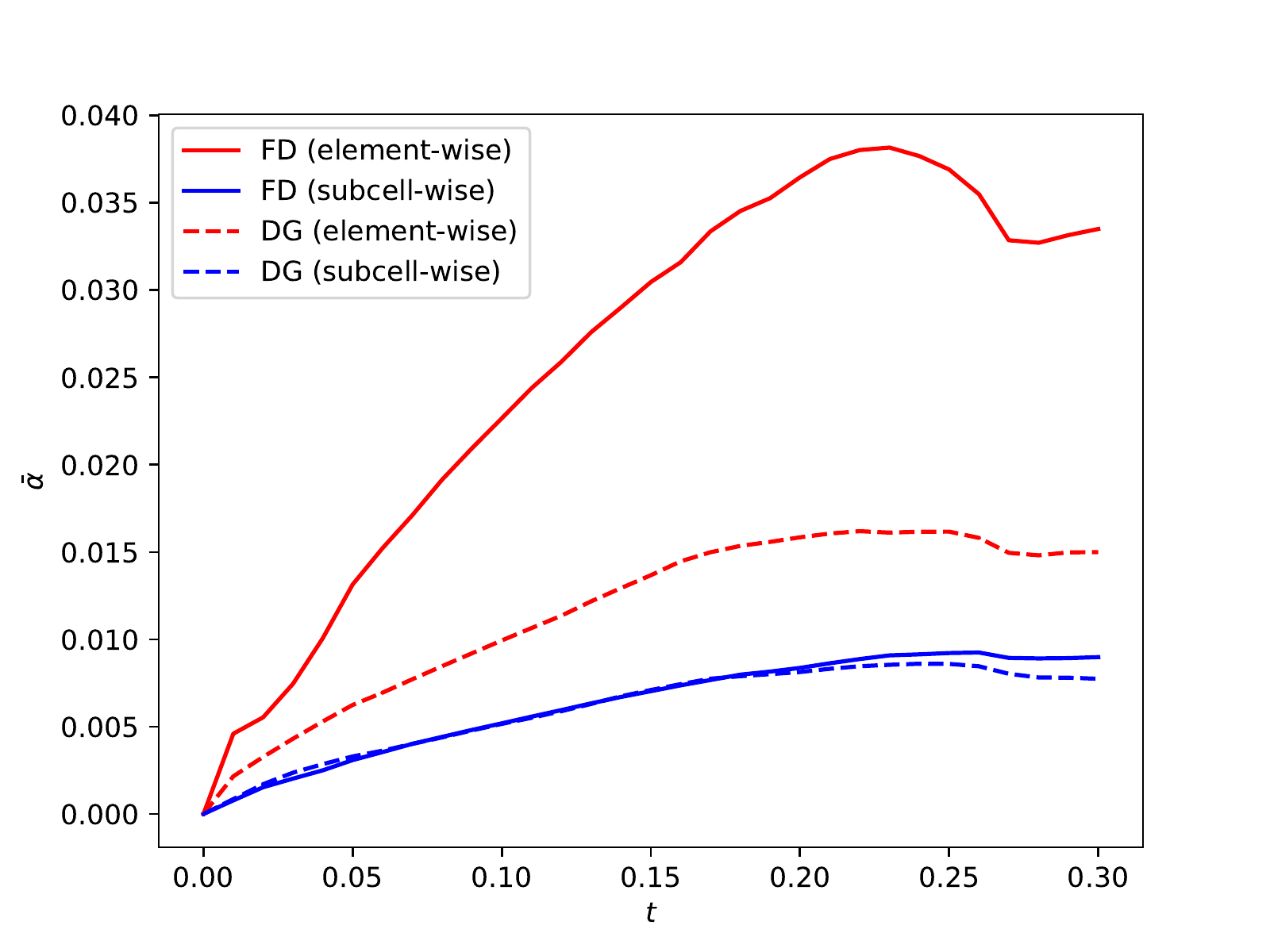}
	\includegraphics[trim=0 0 0 0 ,clip,width=0.45\linewidth]{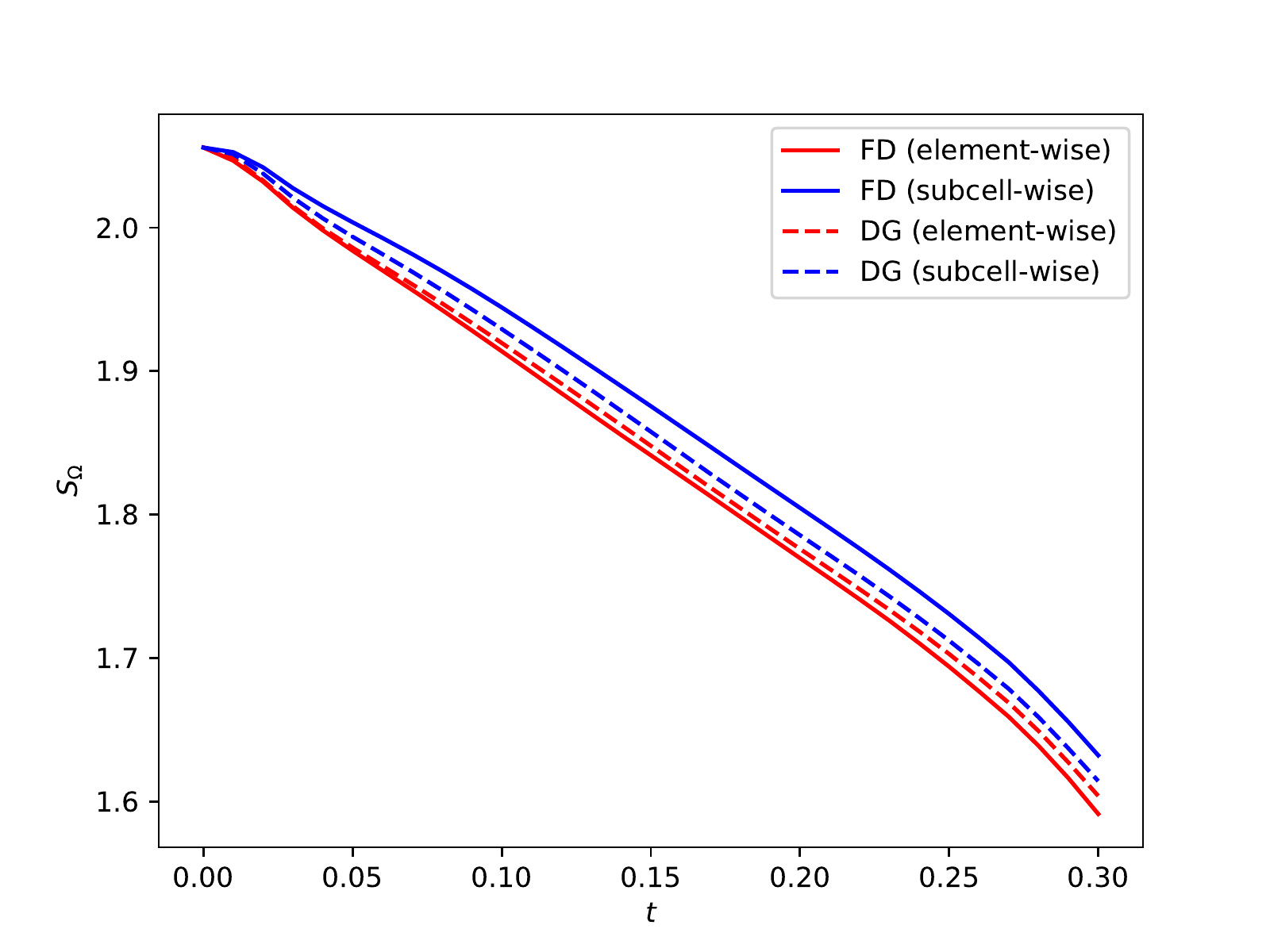}
	\caption{Evolution of the mean blending coefficient and the total entropy for the Orszag-Tang vortex test using the \textit{a-priori} limiting based on \eqref{eq:Loehner}.}
	\label{fig:Rotor_Loehner/Rotor_alpha_ent}
\end{figure}

\section{Conclusions} \label{sec:Conclusions}

We showed that high-order discretizations of non-conservative systems using diagonal-norm summation-by-parts (SBP) operators can rewritten as a difference of so-called staggered (or telescoping) ``fluxes''.
Our staggered ``fluxes'' reduce to the symmetrix and unique telescoping fluxes of \citet{Fisher2013a} for conservative systems, but they are also valid for non-conservative systems.
For non-conservative systems, our staggered ``fluxes'' are in general neither unique nor symmetric for two pairs of contiguous subcells, which accounts for the non-conservative nature of the equations.

We further showed that the existence of a flux-differencing formula enables to combine high-order SBP discretization of non-conservative systems with a low-order FV scheme at the subcell level to improve the robustness of the high-order discretizations.
The subcell-wise limiting is more local, and hence adds less dissipation, than element-wise limiting procedures in the literature.

We demonstrated the advantages of subcell-wise limiting for two different high-order SBP discretization techniques applied to the magnetohydrodynamics equations: the discontinuous Galerkin spectral element method (DGSEM) with Legendre-Gauss-Lobatto points and a high-order SBP finite difference method (SBP-FD).
For both high-order discretizations, subcell-wise limiting delivered better accuracy than element-wise limiting for bith \textit{a-priori} and \textit{a-posteriori} selection of the limiting coefficients.

\section*{Acknowledgments}

Gregor Gassner and Andrés M. Rueda-Ramírez acknowledge funding through the Klaus-Tschira Stiftung via the project ``HiFiLab''.
Gregor Gassner further acknowledges funding by the German Research Foundation under the grant number DFG-FOR5409.

We furthermore thank the Regional Computing Center of the University of Cologne (RRZK) for providing computing time on the High Performance Computing (HPC) system ODIN as well as support.

\printcredits

\bibliographystyle{model1-num-names}

\bibliography{Biblio.bib}

\section*{Appendices}
\appendix

\section{Extension to 2D and 3D Discretizations on Curvilinear Meshes} \label{app:3D}

The extension of system \eqref{eq:noncons-system} to multiple dimensions reads
\begin{equation} \label{eq:noncons-system_3D}
\partial_t \mathbf{u} 
+ \Nabla \cdot \blocktensor{f}^a (\mathbf{u}) 
+ \noncon(\mathbf{u}, \Nabla \mathbf{u})
= \state{0},
\end{equation}
where the non-conservative term can be expressed in different forms.
For instance, it can depend on the divergence of a quantity, i.e.,
\begin{equation}
    \noncon(\mathbf{u}, \Nabla \mathbf{u}) = \stateG{\phi} (\state{u}) ( \Nabla \cdot {\vec{b}}),
\end{equation}
like the Godunov-Powell non-conservative term, for which $\vec{b}$ corresponds to the magnetic field, or it can depend on the gradient of a quantity, i.e.,
\begin{equation}
    \noncon(\mathbf{u}, \Nabla \mathbf{u}) = \blocktensorG{\phi} (\state{u}) \cdot \Nabla b,
\end{equation}
like the GLM non-conservative term, for which $b$ corresponds to the divergence cleaning field, or
\begin{equation}
    \vec{\noncon}(\mathbf{u}, \Nabla \mathbf{u}) = \phi (\state{u}) \Nabla b,
\end{equation}
like in the case of the shallow water equations, for which $b$ corresponds to the bottom topography.

Regardless of the form of the non-conservative term, a high-order SBP discretization on multiple space dimensions, $d>1$, and curvilinear meshes can be constructed using tensor-product expansions of \eqref{eq:DGSEM}.
The operators are commonly constructed such, that the differentiation and integration operations \eqref{eq:SBP_diff_int} are done in a reference element, $\vec{\xi} \in [-1,1]^d$, and a high-order mapping is defined to convert from reference to physical space, $\vec{x}=\vec{x}(\vec{\xi})$.
For instance, the 3D high-order SBP discretization reads \cite{rueda2022entropy}
\begin{align} \label{eq:Gauss_DGSEM_GLMMHD_3D}
J_{ijk} \omega_{ijk} \dot{\state{u}}_{ijk} 
&+ 
\omega_{jk} \left\lbrace
  \sum_{m=0}^N \hat{S}_{im} 
    \left(
    \tilde{\state{f}}^{1*}_{(i,m)jk}
    +\numnonconsSxi{\tilde{\Jan}}_{(i,m)jk}
    \right)
 + \delta_{iN} \left(
  \numfluxb{f}_{(N,R)jk}
    +\numnonconsD{\Jan}_{(N,R)jk}
  \right)
 - \delta_{i0} \left(
  \numfluxb{f}_{(0,L)jk}
    +\numnonconsD{\Jan}_{(0,L)jk}
  \right)
\right\rbrace
\nonumber\\
&+ 
\omega_{ik} \left\lbrace
  \sum_{m=0}^N \hat{S}_{jm}
    \left(
    \tilde{\state{f}}^{2*}_{i(j,m)k}
    +\numnonconsSeta{\tilde{\Jan}}_{i(j,m)k}
    \right)
 + \delta_{jN} \left(
  \numfluxb{f}_{i(N,R)k}
    +\numnonconsD{\Jan}_{i(N,R)k}
  \right)
 - \delta_{j0} \left(
  \numfluxb{f}_{i(0,L)k}
    +\numnonconsD{\Jan}_{i(0,L)k}
  \right)
\right\rbrace
\nonumber\\
&+ 
\omega_{ij} \left\lbrace
  \sum_{m=0}^N \hat{S}_{km}
    \left(
    \tilde{\state{f}}^{3*}_{ij(k,m)}
    +\numnonconsSzeta{\tilde{\Jan}}_{ij(k,m)}
    \right)
 + \delta_{kN} \left(
  \numfluxb{f}_{ij(N,R)}
    +\numnonconsD{\Jan}_{ij(N,R)}
  \right)
 - \delta_{k0} \left(
  \numfluxb{f}_{ij(0,L)}
    +\numnonconsD{\Jan}_{ij(0,L)}
  \right)
\right\rbrace
= \state{0},
\end{align}
where $J_{ijk}$ is the determinant of the mapping Jacobian, which may now be different at each degree of freedom of the element, and the contravariant basis vectors, $\vec{a}^m_{ijk}=\Nabla \xi^m$, define the mapping from reference space to physical space.
Moreover, the two- and three-dimensional quadrature weights are defined from the one-dimensional weights as
\begin{equation}
\omega_{ij} := \omega_i \omega_j, ~~~~~ \omega_{ijk} := \omega_i \omega_j \omega_k.
\end{equation}
The volume numerical two-point fluxes are
\begin{align}
\tilde{\state{f}}^{1*}_{(i,m)jk} &:= \blocktensor{f}^{*}(\state{u}_{ijk}, \state{u}_{mjk}) \cdot \avg{J\vec{a}^1}_{(i,m)jk}, ~~~~~
\tilde{\state{f}}^{2*}_{i(j,m)k} := \blocktensor{f}^{*}(\state{u}_{ijk}, \state{u}_{imk}) \cdot \avg{J\vec{a}^2}_{i(j,m)k}, \nonumber \\
\tilde{\state{f}}^{3*}_{ij(k,m)} &:= \blocktensor{f}^{*}(\state{u}_{ijk}, \state{u}_{ijm}) \cdot \avg{J\vec{a}^3}_{ij(k,m)},
\end{align}
where $\blocktensor{f}^{*}$ is a symmetric and consistent two-point averaging flux function, which can be selected to provide entropy conservation \cite{Derigs2017}, $\avg{\cdot}$ denotes the average operator, and the multiplication with the average of $J \vec{a}$ is a technique commonly known as metric dealiasing \cite{Gassner2016}.
The non-conservative two-point terms, $\tilde{\Jan}^{d\star}$, also perform metric dealiasing, but not necessarily with the average of the metric terms \cite{rueda2022entropy}.

Following the procedure of Section~\ref{sec:fluxdiff}, it is possible to rewrite \eqref{eq:Gauss_DGSEM_GLMMHD_3D} as a flux-differencing formula,
\begin{equation} \label{eq:fluxdiff3d}
    J_{ijk} \dot{\state{u}}_{ijk} = 
        \frac{1}{\omega_i}
        \left(
        \stateG{\Gamma}^{1}_{(i,i-1)jk}
        -\stateG{\Gamma}^{1}_{(i,i+1)jk}
        \right)
        +
        \frac{1}{\omega_j}
        \left(
        \stateG{\Gamma}^{2}_{i(j,j-1)k}
        -\stateG{\Gamma}^{2}_{i(j,j+1)k}
        \right)
        +
        \frac{1}{\omega_k}
        \left(
        \stateG{\Gamma}^{3}_{ij(k,k-1)}
        -\stateG{\Gamma}^{3}_{ij(k,k+1)}
        \right)
    ,
    \qquad
    j=0, \ldots, N,
\end{equation}
where the indexes $j=-1$ and $j=N+1$ refer to the outer states and $\stateG{\Gamma}^d_{(a,b)}$ is the so-called staggered (or telescoping) ``flux'' between node $a$ and the \textbf{adjacent} node $b$.

As in Section~\ref{sec:fluxdiff}, the only requirement of \eqref{eq:fluxdiff3d} is that the volume numerical two-point terms can be written as a product of local and symmetric contributions, e.g.,
\begin{equation} \label{eq:condition3d}
    {\numnonconsSeta{\tilde{\Jan}}}_{i(j,m)k} := \Jan^{2,\mathrm{loc}}_{ijk} \circ \, \Jan^{2,\mathrm{sym}}_{i(j,m)k},
\end{equation}
where $\Jan^{2,\mathrm{loc}}_{ijk} := \Jan^{2,\mathrm{loc}} (\state{u}_{ijk},\state{b}_{ijk})$ only depends on local quantities, and  $\Jan^{2,\mathrm{sym}}_{i(j,m)k} := \Jan^{2,\mathrm{sym}}_{i(m,j)k}$ is a symmetric two-point flux that depends on values at nodes $ijk$ and $imk$.

The staggered fluxes are defined as in \eqref{eq:leftFlux}-\eqref{eq:rightFlux} in each reference coordinate direction.
For instance, the staggered fluxes in $\eta$-direction read
\begin{align}
\stateG{\Gamma}^2_{i(0,-1)k}  &= \numfluxb{f}_{i(0,L)k} + \numnonconsD{\Jan}_{i(0,L)k}
\label{eq:leftFlux3d}
\\
\stateG{\Gamma}^2_{i(j,m)k} &= \sum_{l=0}^{\min(j,m)} \sum_{n=0}^N S_{ln} \tilde{\state{f}}^{2*}_{i(l,n)k} 
+ \Jan^{2,\mathrm{loc}}_{ijk} \circ \sum_{l=0}^{\min(j,m)} \sum_{n=0}^N S_{ln} \Jan^{2,\mathrm{sym}}_{i(l,n)k}, & j=0, \ldots, N-1, m \in \{j-1,j+1\},
\label{eq:internFlux3d}\\
\stateG{\Gamma}^2_{i(N,N+1)k} &= \numfluxb{f}_{i(N,R)k} +  \numnonconsD{\Jan}_{i(N,R)k}.
\label{eq:rightFlux3d}
\end{align}

The proof follows exactly as the proof of Proposition~\ref{prop:fluxdiff}.

\end{document}